%% file: template.tex
\documentclass{article}

\usepackage{arxiv}

\usepackage[utf8]{inputenc} 
\usepackage[T1]{fontenc}    
\usepackage[colorlinks,citecolor=blue,urlcolor=blue,linkcolor=magenta]{hyperref}
\usepackage{url}            
\usepackage{booktabs}       
\usepackage{nicefrac}       
\usepackage{microtype}      
\usepackage{lipsum}		
\usepackage{graphicx}
\usepackage{natbib}
\usepackage{doi}

\usepackage{mathrsfs}
\usepackage{amsthm,amssymb,amsfonts}
\usepackage{mathtools}
\usepackage{array}
\usepackage{bbm}
\usepackage{bm,upgreek}
\usepackage{enumerate}
\usepackage{caption} 
\usepackage{subcaption}
\usepackage{longtable}
\usepackage{indentfirst}
\usepackage[mathscr]{eucal}
\usepackage{xifthen}
\usepackage{nomencl}
\usepackage{comment}
\usepackage{tagging}
\usepackage[dvipsnames]{xcolor}
\usepackage{systeme}

\usepackage{mathrsfs}%
\usepackage[title]{appendix}%
\usepackage{textcomp}%
\usepackage{manyfoot}%
\usepackage{listings}%

\mathtoolsset{showonlyrefs}
\input{mydef}

\newtheorem{theorem}{Theorem}[section]
\newtheorem{proposition}[theorem]{Proposition}%

\newtheorem{corollary}[theorem]{Corollary}
\newtheorem{lemma}[theorem]{Lemma}

\newtheorem{remark}[theorem]{Remark}%

\title{Bernstein-type and Bennett-type inequalities for unbounded matrix martingales}

\author{\hspace{1mm}Alexey Kroshnin \\
    Weierstrass Institute for\\
    Applied Analysis and Stochastics\\
	HSE University\\
	\texttt{kroshnin@wias-berlin.de} \\
	\And
    \hspace{1mm}Alexandra Suvorikova \\
	Weierstrass Institute for\\
    Applied Analysis and Stochastics\\
	HSE University\\
	\texttt{suvorikova@wias-berlin.de} \\
}

\renewcommand{\shorttitle}{Bernstein-type and Bennett-type inequalities for matrix martingales}

\hypersetup{
pdftitle={Bernstein and Bennett-type inequalities for unbounded matrix martingales},
pdfsubject={math.PR, math.ST},
pdfauthor={Alexey Kroshnin, Alexandra Suvorikova},
pdfkeywords={concentration, Bernstein inequality, Bennett inequality, martingales, Orlicz norm, empirical Bernstein inequality, McDiarmid inequality},
}

\begin{document}
\maketitle

\begin{abstract}
    We derive explicit Bernstein-type and Bennett-type concentration inequalities for matrix-valued martingale processes with unbounded observations from the Hermitian space $\H(d)$. Specifically, we assume that the $\pa$-Orlicz~(quasi-)norms of their difference process are bounded for some $\alpha > 0$. Further, we generalize the obtained result by replacing the ambient dimension $d$ with the effective rank of the covariance of the observations. To illustrate the applicability of the results, we prove several corollaries, including an empirical version of Bernstein's inequality and an extension of the bounded difference inequality, also known as McDiarmid's inequality.
\end{abstract}


\section{Introduction}\label{sec:intro}

Non-asymptotic concentration inequalities play an essential role in a wide variety of fields, including probability theory, statistics \citep{arcones1995bernstein}, graph theory \citep{krebs2018bernstein}, machine learning \citep{lopez2014randomized}, theoretical computer science \citep{tolstikhin2013pac}, quantum statistics \citep{girotti2023concentration}, etc. 
These inequalities provide crucial probabilistic bounds that facilitate rigorous analysis in both theoretical and applied contexts. 
Key references for comprehensive surveys include the works by \citet{ledoux1991probability, koltchinskii2011oracle, boucheron2013concentration}.

This paper explores Bernstein-type and Bennett-type inequalities, which are pivotal in various research domains. 
These concentration inequalities play a crucial role in the analysis of weakly dependent observations \citep{merlevede2009bernstein, banna2016bernstein}, martingales \citep{dzhaparidze2001bernstein, tropp2011freedman}, stochastic and empirical processes \citep{bechar2009bernstein, baraud2010a, hanyuan2017a}, and the concentration of matrices and operators \citep{mackey2014matrix, minsker2017some}.

\subsection{Related works}

In the literature survey, we aim to highlight the significant milestones in the development of Bernstein-type bounds. 
The survey is structured chronologically, providing a comprehensive understanding of the field's evolution.

\paragraph{Scalar case.} 
The early results, dating back to the beginning of the 20th century, deal mainly with bounded observations. 

The celebrated Bernstein's inequality---formulated by Sergei Bernstein in the late 1920s \citep{bernvstein1927theory}---stands as a cornerstone in the theory of concentration inequalities. 
It guarantees an exponential decay rate for the tail probabilities of the sum of independent bounded random variables.
 
\begin{proposition}[Bernstein's inequality (bounded case)]
\label{prop:Bernst}
    Let $X_1, \ldots, X_n$ be independent random variables such that
    \[
    \E X_i = 0, 
    \quad
    X_i \le K~\text{a.s.,}
    \quad
    \sigma^2 \eqset \sum_{i=1}^n \E X^2_i.
    \]
    Then for all $t > 0$
    \[
    \P\lr{\sum_{i=1}^n X_i \ge t} \le 
    \elr{- \frac{t^2}{2 \left(\sg^2 + \frac{K t}{3}\right)}}.
    \]
\end{proposition}

Note that Bernstein also proposed going beyond the bounded case and considered the following moment bounds,
\begin{equation}
\label{def:moment_condition}
    \E X^p_i \le \frac{p!}{2} U^{p-2}_i \sg^2_i, \quad p = 2, 3, \dots
\end{equation}
with $\sigma^2_i = \E X^2_i$, and $U_i > 0$ being some constant.
This assumption is now known as Bernstein's moment condition. 
It ensures the sub-gamma behavior of $X_i$, see Corollary~5.2 in \cite{zhang2020concentration}. 
However, further research on the unbounded case did not attract much attention until the beginning of the 21st century. 

The next famous result concerning bounded observations---derived by George Bennett in 1962 \citep{bennett1962probability}---presents a sharper version of Proposition~\ref{prop:Bernst}.
 
\begin{proposition}[Bennett's inequality]
\label{prop:Benn}
    Under assumptions of Proposition~\ref{prop:Bernst}, it holds for all $t > 0$ that
    \[
    \P\lr{\sum_{i=1}^n X_i \ge t} \le \elr{- \frac{\sg^2}{K^2} h\lr{\frac{K t}{\sg^2}}},
    \]
    where $h(x) \eqset (1 + x) \ln(1 + x) - x$ for all $x\ge 0$.
\end{proposition}
 
Alongside independent observations, the dependent case also gained attention. 
So, in 1975, David Freedman \citep{freedman1975tail} derived the famous martingale extension of Proposition~\ref{prop:Bernst}.

Almost in parallel, in 1976, Vadim Yurinskii generalized Proposition~\ref{prop:Bernst} to the case of random variables in Banach spaces. 
He assumed the norm of observations to satisfy Bernstein's moment condition~\eqref{def:moment_condition}  \citep{yurinsky1976exponential}.

\paragraph{Matrix case.} Joel Tropp, in 2011, generalized Freedman's result to the case of matrix-valued martingales (see Proposition~\ref{prop:tropp_bound}). 
One year later, he got a result similar to Yurinsky's one. 
Namely, he applied assumption~\eqref{def:moment_condition} to the matrix-valued case (see Proposition~\ref{prop:tropp_moment}).

In the following, we denote as $\lmax(\X)$ the largest eigenvalue of $\X$, as $\norm{X}$ the operator norm of $X$, and as $\H(d)$ the space of $d \times d$ Hermitian matrices.
 
\begin{proposition}[Theorem 3.1, \cite{tropp2011freedman}]
\label{prop:tropp_bound}
    Let $\X_1, \dots, \X_n \in \H(d)$ be a sequence adapted to a filtration $\F_1 \subset \F_2 \subset \dots \subset \F_n $. 
    Let $\X_k$ satisfy for all $k = 1, \dots, n$
    \[
    \E[\X_k | \F_{k-1}] = 0, 
    \quad
    \lmax(\X_k) \le K~\text{a.s.}
    \]
    Set for all $k = 1, \dots, n$
    \[
    \Sm_k  \eqset \sum^k_{i=1} \X_i, 
    \quad
    \Sg_k \eqset \sum^k_{i=1}  \E\left[\X^2_i \middle| \F_{i-1}\right].
    \]
    Then, for all $t \ge 0$ and $\sg^2 > 0$,
    \[
    \P \lr{\exists k \ge 0: \lmax(\Sm_k) \ge t~~\text{and}~~ \norm{\Sg_k} \le \sg^2} \le d \elr{-\frac{\sg^2}{K^2} h\left(\frac{K t}{\sg^2} \right)}.
    \]
\end{proposition}

\begin{proposition}[Theorem~6.2, \cite{tropp2012user}]
\label{prop:tropp_moment}
    Let independent random matrices $\X_1, \dots, \X_n \in \H(d)$ satisfy for all $k = 1, \dots, n$, all $p = 2, 3, \dots$, with some $K > 0$ and positive-definite matrices $\Sg_k$
    \[
    \E \X_k = 0,
    \quad
    \E \X^p_k \preccurlyeq \frac{p!}{2} K^{p-2} \Sg_k.
    \]
    Let $\sg^2 \eqset \norm*{\sum^n_{k = 1} \Sg_k}. $ Then for all $t \ge 0$
    \[
    \P \lr{\lmax\left(\sum^n_{k=1} \X_k\right) \ge t} \le d \elr{- \frac{t^2}{2 (\sg^2 + K t)}}.
    \]  
\end{proposition}

All the above results deal with bounded random variables or those satisfying Bernstein's moment condition. 
As one can see, the moment condition is too strong, especially in the case of random matrices. 

The current study focuses on the unbounded case. 
To introduce the setting, we briefly recall the concept of the Orlicz norm.
The Orlicz function we use is
\[
\pa(x) \eqset e^{x^{\alpha}} - 1, \quad \alpha > 0.
\]
The $\pa$-Orlicz (quasi-)norm of a real-valued random variable $X$ is
\begin{equation}\label{def:orlicz}
    \norm{X}_{\pa} \eqset \inf\lrc{t > 0 : \E \pa\lr{\frac{\abs{X}}{t}} \le 1 }.
\end{equation}

If $\alpha \ge 1$, $\norm{X}_{\pa}$ is a norm. 
In particular, if $\norm{X}_{\psi_{1}} < \infty$, $X$ is sub-exponential, and if $\norm{X}_{\psi_{2}} < \infty$, $X$ is sub-Gaussian. 
Moreover, if $0 < \alpha < 1$, $\norm{X}_{\pa}$ is a quasi-norm. 

In 2008, Radoslaw Adamczak got the concentration result for unbounded empirical processes \citep{adamczak2008tail}. 
Being applied to a particular case of a sum of independent observations, it yields a Bernstein-type deviation bound. 
The result holds under the assumption that the summands have finite $\pa$-Orlicz (quasi-)norm for $0 < \alpha \le 1$.

In 2011, Vladimir Koltchinskii obtained an extension of Proposition~\ref{prop:Bernst} for a sum of independent Hermitian matrices with bounded $\pa$-Orlicz norm.
 
\begin{proposition}[Theorem~2.7, \cite{koltchinskii2011oracle}]
\label{prop:koltch_indep}
    Let $\X_1, \dots, \X_n \in \H(d)$ be independent random matrices.
    Fix $\alpha \ge 1$. 
    Suppose, for all $k = 1, \dots, n$, and some $K > 0$,
    \begin{equation}
    \label{def:K_koltchinskii}
        \E \X_k = 0, 
        \quad
        \max \left(\norm*{\norm{\X_k}}_{\psi_{\alpha}}, 2 \sqrt{\E \norm{\X_k}^2}\right) 
        \le K.
    \end{equation}
    Set $  \sg^2 \eqset \norm*{\sum_{k=1}^n \E \X^2_k}.$ Then, there exists an absolute constant $C > 0$ such that, for all $t \ge 0$,
    \[
    \P\left(\norm*{\sum_{k=1}^n \X_k} \ge t\right) 
    \le 2 d \exp\lrc{-\frac{1}{C} \frac{t^2}{\sg^2 + t K \left(\log \frac{n K^2}{\sg^2}\right)^{1/\alpha}} }.
    \]
\end{proposition}

\paragraph{Dimension-free bounds.} Modern advances in statistics and machine learning often require more efficient bounds in terms of dimension. Specifically, the informative signal in the data typically lies in a low-dimensional subspace. In this case, it is more beneficial to use bounds depending not on the ambient dimension of the observations $\X \in \H(d)$ but on the intrinsic dimension of variance $\Sg \in \H_+(d)$, 
\[
r(\Sg) \eqset \frac{\tr \Sg}{\norm*{\Sg}}.
\]
In some works, $r(\Sg)$ is referred to as the effective rank. 
A comprehensive introduction to the topic can be found in Chapter 7 by \cite{tropp2012user_nips}. In particular, Theorem~7.3.1 introduces matrix Bernstein's inequality in the bounded case, where the dimension $d$ is replaced (up to a multiplicative constant) with the intrinsic dimension $r(\Sg)$.

In 2017, Stanislav Minsker derived a dimension-free version of matrix Freedman's inequality.
\begin{proposition}[Theorem 3.2, \cite{minsker2017some} ]
\label{prop:Minsker}
    Let $\X_1, \dots, \X_n \in \H(d)$ be a sequence of martingale differences such that $\max_{i} \norm{\X_i} \le K$ almost surely for some $K > 0$. Denote $\Sg \eqset \sum^n_{i=1}  \E\left[\X^2_i \middle| \F_{i-1}\right].$ Then for any $t \ge \frac{1}{6}(K + \sqrt{K^2 + 36\sg^2})$,
    \[
    \P\lr{\norm*{\sum^n_{i=1} \X_i} > t, ~\lmax(\Sg) \le \sg^2}
    \le 50\tr\lr{\min\lrc{1, \frac{t}{K}\frac{\E\Sg}{\sg^2}}}\exp\lr{-\frac{t^2/2}{\sg^2 + \frac{tK}{3}}}.
    \]
\end{proposition}
Several years later, Egor Klochkov and Nikita Zhivotovskiy got a dimension-free Bernstein-type inequality for sub-exponential case. The proof technique relies on the tools from \cite{minsker2017some} and \cite{adamczak2008tail}. 
\begin{proposition}[Proposition 4.1, \cite{klochkov2020hanson}]
\label{prop:Klochkov_and_Yhivotovskiy}
    Consider the set of independent Hermitian matrices $\X_1, \dots, \X_n \in \H(d)$ such that $\norm*{\norm{\X_i}}_{\psi_1} \le \infty$. Set $M \eqset \norm*{\max_{i\in [n]}\norm{\X_i}}_{\psi_1}$ and let the positive-definite matrix $\Sg$ be s.t. $\E \sum_i \X^2_i \preccurlyeq \Sg$. Finally, set $\sg^2 \eqset \norm{\Sg}$. There are absolute constants $c, C, c_1 > 0$ such that for any $t>c_{1}\max\{M, \sg\}$ we have
    \[
    \P \lr{\norm*{\sum^n_{i=1} \X_i - \E\X_i} > t} \le C r(\Sg)
    \exp\lr{-c \min\lrc{\frac{t^2}{\sg^2}, \frac{t}{M}}}.
    \]
\end{proposition}

\paragraph{Other results.} Many excellent results deal with the Bernstein-type bounds under different settings, e.g., \citep{van2013bernstein, gao2014bernstein}. The work \cite{shetty2021distribution} introduces sub-Gaussian Freedman matrix inequality (see Lemma 4). The works by  \cite{zhivotovskiy2024dimension, puchkin2023sharper} consider dimension-free bounds.
However, these works are beyond the scope of the current study. 
For other references, we recommend \cite{boucheron2013concentration}.

\subsection{Contribution of the current study}\label{subsec:main_res}

In this work, we consider a matrix-valued supermartingale difference sequence $\bm{0} \equiv \X_0, \X_1, \dots, \X_n \in \H(d)$ adapted to a filtration $(\F_i)^n_{i=0}$ ($\F_0 \eqset \{\Omega, \varnothing\}$ is the trivial $\sg$-algebra), i.e., \ $\E \norm{\X_i} < \infty$ and $\E\lrs{\X_i \middle| \F_{i-1}} \preccurlyeq \bm{0}$ for all $i = 1, \dots, n$. 
Set for any $k=1, \dots, n$
\[
\Sm_k \eqset \sum_{i=1}^k \X_i.
\]
Clearly, $(\Sm_i)_{i=1}^n$ is a matrix supermartingale adapted to $(\F_i)^n_{i=1}$.

\paragraph{Results incorporating ambient dimension.} Theorem~\ref{thm:martingale} shows that one can obtain a combination of Bernstein- and Bennett-type deviation bounds on $\max_{k \in [n]} \lmax(\Sm_{k})$. All constants are computed explicitly. We note that the result can be extended to the case of rectangular matrices using dilations (see, e.g., \citep{paulsen2002completely,tropp2012user}).

The validity of the result depends on assumptions about the behavior of the observations $\X_i$. 
Specifically, we assume the \emph{conditioned Orlicz norm}, $\norm{\lmax(\X_i)_{+} | \F_{i-1}}_{\pa}$, is bounded. 
As far as we know, this is a new concept, so we explain it below.

\textbf{\textit{Conditioned Orlicz norm.}} 
Let $(\Omega, \F, \P)$ be a fixed probability space. 
We denote by $\ind[E]$ the indicator of an event $E \in \F$. 
Given a sub-$\sigma$-algebra of events $\mathscr{G} \subset \F$ and a random variable $X \in \R$, let $\mu_{X | \mathscr{G}}$ be a conditional distribution of $X$ w.r.t.\ $\mathscr{G}$; i.e.\ $\mu_{X | \mathscr{G}}$ is a $\mathscr{G}$-measurable random measure on $\R$ such that for any Borel set $A \subset \R$
\begin{equation}
\label{def:mu_cond}
\mu_{X | \mathscr{G}}(A) = \P\lrs{X \in A \middle| \mathscr{G}} \eqset \E\lrs{\ind[X \in A] \middle| \mathscr{G}} ~~\text{a.s.},
\end{equation}
see \citep[Chapter~2~\S7]{shiryaev2016probability}. We define the conditional Orlicz norm of $X$ as the norm of the conditional distribution $\mu_{X | \F}$, 
\[
\norm*{X | \F}_{\pa} \eqset \norm*{\mu_{X | \F}}_{\pa},
\]
i.e.,\ this is a $\mathscr{G}$-measurable random variable $\omega \mapsto \norm{\mu_{X | \F}(\omega)}_{\pa}$. 
The conditional Orlicz norm of $X$ can be explicitly written, e.g., as
\begin{equation*}
\label{def:cond_norm}
    \norm*{X | \F}_{\pa} = \sup_{t \in \mathbb{Q},\, t \ge 0} t \ind\lrs{\E\left[\psi_\alpha\left(\frac{|X|}{t}\right) \middle| \F\right] > 1 }~~\text{a.s.},
\end{equation*}
where $\mathbb{Q}$ is the set of rational numbers.
Since $\mathbb{Q}$ is countable and dense in $\R$, one can see that this is indeed a random variable and it coincides with $\norm{\mu_{X | \F}(\omega)}_{\pa}$ a.s. 

\paragraph{Dimension-free bounds.}
We derive Bernstein- and Bennett-type bounds incorporating the effective rank $r(\Sg)$ instead of the ambient dimension $d$. Specifically, Theorem~\ref{thm:dim_free_bounded} deals with the bounded observations, and Theorem~\ref{thm:dim_free_unbounded} contains the result for the unbounded case.

\paragraph*{Corollaries.} To demonstrate the practical applicability of the results, we derive several corollaries including an empirical Bernstein inequality and a version of McDiarmid inequality. 

\textbf{\textit{Empirical Bernstein-type inequality.}} Bernstein-type bounds rely on the true variance of the observations, while empirical Bernstein-type bounds, in contrast, incorporate a data-driven variance estimator \citep{peel2010empirical, martinez2024empirical}. 
The latter ones play an essential role in the theoretical analysis of machine learning algorithms \citep{audibert2007tuning, mnih2008empirical, maurer2009empirical, shivaswamy2010empirical, tolstikhin2013pac}. 

Corollary \ref{cor:empirical} presents an empirical Bernstein-type bound for the case of i.i.d.\ matrix-valued observations. 
To the best of our knowledge, this result is novel.

\textbf{\textit{McDiarmid's inequality.}}
McDiarmid's inequality provides a powerful tool for bounding the deviations of functions of independent random variables from their expected values. 
Specifically, it addresses the functions satisfying the bounded property.

Let $\mathcal{Y}$ be a measurable space and let $f\colon \mathcal{Y}^n \rightarrow \R$ be such that there exist $U_1, \dots U_n \in \R_{+}$ satisfying
\[
\sup_{y'_i \in \mathcal{Y}} \left|f(y_1, \dots, y_{i-1}, y_i, y_{i+1}, \dots, y_n) -  f(y_1, \dots, y_{i-1}, y'_i, y_{i+1}, \dots, y_n)\right| \le U_i
\]
for all $y \in \mathcal{Y}^n$.

McDiarmid's inequality \citep[see][(1.3)]{mcdiarmid1989method} ensures that if $Y_1, \dots, Y_n \in \mathcal{Y}$ are independent and if $f$ satisfies the above properties than
\[
\P\left(f(Y_1,\dots, Y_n )  - \E f(Y_1,\dots, Y_n ) \ge t \right) \le \elr{- \frac{2 t^2}{\sum_i U^2_i}}.
\]

Many works develop McDiarmid inequalities under extended settings \citep{kutin2002extensions, rio2013on,warnke2016method, zhang2019mcdiarmid}. 
Among the recent results on the concentration of dependent and unbounded observation, one should mention the work by \citet{maurer2021concentration}. 
The authors propose a Bernstein-type generalization of McDiarmid's inequality for functions with sub-exponential differences.

Corollary \ref{corollary_maurer1} and Corollary \ref{corollary:maurer2} present Bernstein-type McDiarmid inequalities for functions whose differences have bounded $\pa$-Orlicz norm. 
We compare the results with those by \citet{maurer2021concentration}.

\subsection*{Organization of the paper}

Section~\ref{sec:main} presents the main results, which include Bennett- and Bernstein-type bounds stated in terms of both the ambient dimension 
$d$ and the effective rank of the covariance matrix of the observations. This section also examines the tail behavior of the Bennett-type bound and compares the results with those discussed in the Introduction. Section~\ref{sec:corollaries} contains all corollaries. Finally, Section~\ref{sec:proofs} collects the main proofs, and the Appendix provides additional auxiliary results.

\subsection*{Accepted notations}

\textbf{\textit{Spaces and sets.}} We denote as $\H(d)$ the space of all $d$-dimensional Hermitian matrices. 
$\H_{+}(d) \subset \H(d)$ is the set of positive semi-definite Hermitian matrices. 
$\H_{++}(d) \subset \H(d)$ is the set of positive-definite Hermitian matrices. 
Further, we denote the integer indices as $[n] = \{1, \dots, n\}$.

\textbf{\textit{Norms.}} Let $\norm{\bm{A}}$ be the operator norm of a matrix $\bm{A}$. By analogy with the conditioned Orlicz norm, we define the conditioned $L^{\infty}$-norm as $\norm*{X | \F}_{\infty} \eqset \norm*{\mu_{X | \F}}_{\infty}$, with $\mu_{X | \F}$ defined in \eqref{def:mu_cond}.

\textbf{\textit{Functions.}} From now on, we set for any $x \in \R$
\[
\logg x \eqset \max(\ln x, 1), 
\quad
x_{+} \eqset \max(x, 0).
\]

Let $\ind[E]$ be the indicator of an event $E$ (i.e., $\ind[E]$ is a random variable). Respectively, $\ind_E$ denotes the indicator function of a set $E$.

Further, we define functions $\phi \colon \R \to \R$ and $h \colon (-1, \infty) \to \R$ as
\begin{equation}\label{def:phi_h}
    \phi(t) \eqset e^t - 1 - t,
    \quad
    h(x) \eqset (1 + x) \ln(1 + x) - x.
\end{equation}
Note that $h$ is the convex conjugate of $\phi$. Now, let $f$ be a scalar function. 
For any $d \times d$ diagonal matrix $\bm{\Lambda} = \diag(\lm_1, \dots, \lm_d)$, we define
\[
f(\bm{\Lambda}) \eqset \diag\bigl(f(\lm_1), \dots, f(\lm_d)\bigr).
\]
Respectively, given a matrix $\mathbf{A} \in \H(d)$ with a spectral decomposition $\mathbf{A} = \mathbf{U} \bm{\Lambda} \mathbf{U}^*$, we set
\[
f(\mathbf{A}) \eqset \mathbf{U} f(\bm{\Lambda}) \mathbf{U}^*.
\]
We also recall the transfer rule. 
Let for any $x \in I \subset \R$, $f(x) \le g(x)$. 
If all eigenvalues of $\mathbf{A}$ belong to $I$, then $f(\mathbf{A}) \preccurlyeq g(\mathbf{A})$.


\section{Main results}
\label{sec:main}

This section presents the core Bernstein- and Bennett-type bounds for matrix supermartingales. 
Recall that we formulate them using a difference sequence defined in Section~\ref{subsec:main_res}.

\subsection{Bernstein- and Bennett-type bounds with ambient dimension \texorpdfstring{$d$}{d}}

We begin by presenting the bounds involving the ambient dimension $d$.

\begin{theorem}\label{thm:martingale}
    Let $\X_1, \dots, \X_n \in \H(d)$ be a supermartingale difference sequence adapted to a filtration $(\F_i)^n_{i=0}$ with $\F_0 = \{\Omega, \varnothing\}$. Fix $\alpha > 0$ and set for all $i \in [n]$
    \[
    \Sg_i \eqset \E\lrs{\X^2_i \middle| \F_{i-1}},
    \quad
    U_i \eqset \norm*{\lmax(\X_i)_+ \big| \F_{i-1}}_{\pa}.
    \]
    
    Fix $\sg > 0$, $U \ge K > 0$ and define the event
    \begin{equation}
    \label{eq:E}
        E \eqset \lrc{\lmax\left(\sum^n_{i = 1} \Sg_i\right) \le \sg^2, \quad
        \sum^n_{i = 1} U^{2}_i \le U^2, \quad
        \max_{i \in [n]} U_i \le K}.
    \end{equation}
    Let
    \begin{equation}
    \label{eq:def_z}
        z = z(U, \sg; \alpha) \eqset
        \begin{cases}
            \lr{4 \logg \frac{e U}{\sg}}^{1/\alpha} & \text{if}~\alpha\ge 1, \\
            \lrs{\frac{4}{\alpha} \ln \frac{e}{\alpha} + 4 \lr{\ln \frac{U}{\sg}}_+}^{1/\alpha}, &\text{if}~\alpha < 1.
        \end{cases}  
    \end{equation}
    
    Then for any $\xx > 0$, it holds, with probability at least $\P(E) - d e^{-\xx} - e^{-\xx} \ind[\alpha < 1]$, that
    \begin{align}\label{res:ben}
        \max_{k \in [n]} \lmax(\Sm_k) \le& \sg \sqrt{2 \xx} + \frac{4 \kz \xx}{\min\lrc{2 \az, \logg\lr{\lr{\frac{Kz}{\sg}}^2 \xx} }} \nonumber \\
        &+ \frac{3K}{\alpha}\xx\left(2 \xx + 2\ln\lr{\frac{4 U}{K}}
        + \frac{4}{\alpha}\ln \lr{\frac{4}{\alpha e}} \right)^{\frac{1-\alpha}{\alpha}} \ind[\alpha < 1]
        \tag{Ben}
    \end{align}
    and, moreover,
    \begin{align*}\label{res:ber}
        \max_{k \in [n]} \lmax(\Sm_k)  \le \sg \sqrt{2 \xx} + \frac{3}{4} \kz \xx 
        + \frac{3K}{\alpha}\xx\left(2 \xx + 2\ln\lr{\frac{4 U}{K}}
        + \frac{4}{\alpha}\ln \lr{\frac{4}{\alpha e}} \right)^{\frac{1-\alpha}{\alpha}}  \ind[\alpha < 1] .
        \tag{Ber}
    \end{align*}
\end{theorem}
 We postpone the proof to Section~\ref{sec:proofs_main}.
\begin{remark}
\label{rem:non-monotone}
    Note that both bounds can be non-monotone w.r.t.\ $\sg$. 
    Yet, as $\sg$ is just an upper bound, one can improve them. 
    
    For simplicity, we consider the case $\alpha \ge 1$. 
    Setting $z' \eqset \left(4 \logg \frac{e U}{\sg'} \right)^{1/\alpha}$, we get
    \begin{gather}
        \label{eq:ben2} \max_{k \in [n]} \lmax(\Sm_k) \le \inf_{\sg' \ge \sg} \lrc{\sg' \sqrt{2 \xx} + \frac{4 K z' \xx}{\min\lrc{2 \alpha (z')^\alpha ,\; \logg \lr{\lr{\frac{K z'}{\sg'}}^2 \xx} }} }, \nonumber \\
        \max_{k \in [n]} \lmax(\Sm_k) \le \inf_{\sg' \ge \sg} \lrc{\sg' \sqrt{2 \xx} + \frac{3}{4} K z'\xx }.\nonumber
        \label{eq:ber2}
    \end{gather}

The same holds for $0 < \alpha < 1$ with the corresponding substitution of $\sigma$ by $\sigma'$ to $z$.

Moreover, if $\sg' = \sg + K z(U, K \sqrt{\xx}; \alpha) \sqrt{\xx}$, then $z' \le z(U, K \sqrt{\xx}; \alpha) := z$ and thus
\[
\max_{k \in [n]} \lmax(\Sm_k) \le \sg' \sqrt{2 \xx} + \frac{3}{4} K z' \xx \le \sg \sqrt{2 \xx} + \lr{\sqrt{2} z + \frac{3}{4} z'} K \xx \le \sg \sqrt{2 \xx} + \frac{5}{2} K z \xx .
\]
\end{remark}

Notably, the result in \eqref{res:ben} reveals for $\alpha \ge 1$ three distinct regimes---sub-Gaussian, sub-Poisson, and sub-exponential---depending on the magnitude of $\xx$. 

\textbf{\textit{Sub-Gaussian.}} 
If $\kzss \xx \lesssim 1$,
\begin{equation}
\label{eq:subg_r}
    \max_k \lmax(\Sm_k) \lesssim \sg \sqrt{ \xx}.
\end{equation}
We note that if $\kzss \xx \ll 1$, the leading term $\sg \sqrt{2 \xx}$ is sharp: it coincides with the rate in the CLT for martingales (see \citep[Chapter~7 \S8]{Shiryaev2019prob2}).

\textbf{\textit{Sub-Poisson.}} 
If $1 \lesssim \kzss \xx \lesssim e^{\azh}$,
\begin{equation}
    \label{eq:subpoiss_r}
    \max_k \lmax(\Sm_k) \lesssim \frac{\kz \xx}{ \ln\lr{\kzss \xx} }.
\end{equation}

\textbf{\textit{Sub-exponential.}} 
If $\kzss \xx \gtrsim e^{\azh}$,
\begin{equation}
    \label{eq:subexp_r}
    \max_k \lmax(\Sm_k) \lesssim \frac{\kz}{\az} \xx.
\end{equation}

The proofs are postponed to Appendix \ref{seq:proof_r}.

\subsection{Dimension-free bounds}
\label{sec:dim_free}

This section presents dimension-free bounds on $\max_{k\in[n]}\lmax(\Sm_k)$.
The first result provides a Bennett-type inequality for a sum of bounded matrix martingale differences. The proof relies on the methods presented in Section 7 (``Results Involving the Intrinsic Dimension'') by \cite{tropp2012user_nips}.

\begin{theorem}[Bounded case]
\label{thm:dim_free_bounded}
    Let $\X_1, \dots, \X_n \in \H(d)$ be a martingale difference sequence adapted to a filtration $(\F_i)^n_{i=0}$ with $\F_0 = \{\Omega, \varnothing\}$.
    Set $\Sg_i \eqset \E\lrs{\X^2_i \middle| \F_{i-1}}$ for all $i \in [n]$. Fix $K > 0$, $\Sg \succcurlyeq 0$ and define the event
    \[
    E \eqset \lrc{\max_{i\in[n]} \norm*{\lmax(\X_i)_{+} | \F_{i-1} }_{\infty}\le K, \quad \sum^n_{i=1} \Sg_i \preccurlyeq \Sg}.
    \]
    Denote $\sg^2 \eqset \lmax\left(\Sg \right)$.
    Then for all $t > 0$ it holds
    \begin{equation}
        \P\lr{E \bigcap \lrc{\max_{k} \lmax(\Sm_k)\ge t } } \le e \cdot r(\Sg) \elr{-\frac{\sg^2}{K^2}h\lr{\frac{Kt}{\sg^2}} }.
    \end{equation}
    Alternatively, for all $\x \ge 1$, with probability at least $\P(E) - r(\Sg)e^{-\x + 1}$, it holds
    \[
    \max_{k\in [n]}\lmax(\Sm_k) \le \frac{\sg^2}{K} h^{-1}\lr{\frac{K^2}{\sg^2} \x} \le \sg \sqrt{2\x} + \frac{2K\x}{\log\lr{\lr{\frac{K}{\sg}}^2\x}}.
    \]
\end{theorem}

The next result extends Theorem~\ref{thm:dim_free_bounded} to the unbounded case.
The proof technique relies on a combination of several approaches.  First, we decompose the random matrix into its main and tail components using a suitably chosen cut-off threshold. Then, we use Theorem~\ref{thm:dim_free_bounded} to bound the main components and apply Theorem~\ref{thm:martingale} to control the tail components via their norms.

\begin{theorem}[Unbounded case]
    \label{thm:dim_free_unbounded}
    Let $\X_1, \dots, \X_n \in \H(d)$ be a martingale difference sequence adapted to a filtration $(\F_i)^n_{i=0}$ with $\F_0 = \{\Omega, \varnothing\}$. 
    Fix $\alpha > 0$ and set for all $i \in [n]$
    \[
    \Sg_i \eqset \E\lrs{\X^2_i \middle| \F_{i-1}},
    \quad
    U_i \eqset \norm*{\lmax(\X_i)_+ \big| \F_{i-1}}_{\pa}.
    \]
    
    Fix $\Sg \succcurlyeq 0$, $U \ge K > 0$ and define the event
    \begin{equation}
    \label{eq:E_p1}
        E \eqset \lrc{\sum^n_{i = 1} \Sg_i \preccurlyeq \Sg,
        \quad
        \sum^n_{i = 1} U^{2}_i \le U^2,
        \quad
        \max_{i \in [n]} U_i \le K}.
    \end{equation}
    Let $\sg^2 \eqset \lmax(\Sg)$, and $z$ be as in \eqref{eq:def_z}. Then for any $\xx > 0$ with probability at least $\P(E) - (e\cdot r(\Sg) +1) e^{-\xx} - e^{-\xx} \ind[\alpha < 1]$, it holds for any $0 < \eps\le 1$
    \begin{align}\label{res:ben_dim_free}
        \max_{k \in [n]} \lmax(\Sm_k) \le& (1 + \eps)\sg \sqrt{2\xx} + \frac{7 \lr{\ln \frac{8}{\eps}}^{1/\alpha} \kz \xx}{\min\lrc{2 \az, \logg\lr{\lr{\frac{K z}{\sg}}^2 \xx} }} \nonumber \\
        &+ \frac{3 K}{\alpha} \left(2 \xx + 2\ln\lr{\frac{4 U}{K}}
        + \frac{4}{\alpha}\ln \lr{\frac{4}{\alpha e}} \right)^{\frac{1-\alpha}{\alpha}} \ind[\alpha < 1] ,
    \end{align}
    and moreover,
    \begin{align}\label{res:ber_dim_free}
       \max_{k \in [n]} \lmax(\Sm_k) \le& (1 + \eps)\sg \sqrt{2 \xx} + 2 \lr{\ln \frac{8}{\eps}}^{1/\alpha} \kz \xx \nonumber \\
       &+ \frac{3 K}{\alpha} \left(2 \xx + 2\ln\lr{\frac{4 U}{K}}
        + \frac{4}{\alpha}\ln \lr{\frac{4}{\alpha e}} \right)^{\frac{1-\alpha}{\alpha}} \ind[\alpha < 1] .
    \end{align}
\end{theorem}

The proofs of Theorem~\ref{thm:dim_free_bounded} and Theorem~\ref{thm:dim_free_unbounded} are in Section~\ref{sec:proofs_dim_free}.

\subsection{Comparison with other results}
\label{subsec:comparison}

\paragraph*{Bennett-type result for $\alpha \rightarrow \infty$, Proposition~\ref{prop:tropp_bound}.} 

Recall that the classical Bennett's bound corresponds to the case $\alpha = +\infty$ (bounded case).
Consider~\eqref{res:ben}. 
If $\alpha \rightarrow \infty$, then $z \rightarrow 1$. 
This yields $\alpha z^{\alpha} \rightarrow \infty$. 
Thus, one gets the sub-Poisson tail behavior~\eqref{eq:subpoiss_r} that coincides with the Bennett-type bound from Proposition~\ref{prop:tropp_bound} up to a multiplicative constant.

\paragraph*{Bernstein for sub-gamma matrices, Proposition~\ref{prop:tropp_moment}.}
Bernstein's moment condition for scalar random variables is equivalent to
\[
\E \phi\left(\frac{|X_i|}{U_i}\right) \lesssim \frac{\sigma_i^2}{U_i^2},
\]
up to multiplicative constants (see p.103 in \citep{van1996weak}). 
Moreover, the proof of Lemma~\ref{lemma:mgf_X_bound} ensures that a bound on the Orlicz norm yields Bernstein's moment condition. 
This, in turn, yields Bernstein's concentration inequality in the scalar case.

This approach requires two-sided bounds on $X_i$ while Theorem~\ref{thm:martingale} requires only a one-sided one.
Furthermore, in the matrix setting, it is not apparent how to obtain Bernstein's condition for non-isotropic $\Sg_i$ (i.e.,\ for $\Sg_i$ with large condition number), except in the case of bounded or commutative random matrices $\X_i$.

\paragraph*{Adamczak's result ($\alpha \le 1$) \citep{adamczak2008tail}.} 
This work focuses on empirical processes. 
The proof technique primarily builds upon the Klein--Rio bound \citep{klein2005concentration}, Hoffman--J\o rgensen, and Talagrand inequalities. 
As a result, the derived bound includes an additional multiplicative term that arises naturally from these methods and is standard in the setting of empirical processes.

To illustrate the difference between our results and those of Adamczak, we consider a sum of independent scalar random variables $X_1, \dots, X_n$. 
The author uses truncation of $X_i$ at a certain constant level.
This yields a bound on a quantile of the tail term. The bound is proportional to
$\norm*{\max_{i \in [n]} |X_i|}_{\pa} \xx^{1/\alpha}$ (see equation~(11) in \citep{adamczak2008tail}). This bound is comparable to the quantile threshold $\tau$ given by \eqref{def:ok_tau} in the proof of Theorem~\ref{thm:martingale}.

\paragraph*{Koltchinkii's bound ($\alpha \ge 1$), Proposition~\ref{prop:koltch_indep}.}
This setting by Vladimir Koltchinskii is the closest to the current study setting. However, there are several differences. 
First, Theorem~\ref{thm:martingale} handles dependent observations, while Proposition~\ref{prop:koltch_indep} assumes their independence. 
Further, our result is one-sided: \eqref{def:K_koltchinskii} requires $\norm*{\norm{\X_k}}_{\pa}$ to be bounded, while our result relies only on the boundedness of the $\norm*{\lmax(\X_k)_{+}}_{\pa}$. 
Moreover, \eqref{def:K_koltchinskii} depends on $n$, as it uses the upper bound $n K^2$ instead of $U^2$.
Finally, Koltchinskii derives only a Bernstein-type bound, while the current study presents a mixed bound. 

\paragraph*{The result by Minsker, Proposition~\ref{prop:Minsker}.} 
This result has a similar setting as Theorem~\ref{thm:dim_free_bounded}. The difference is that we consider a path-wise maximum deviation of $\max_{k\in[n]}\lmax(\Sm_k)$, while Proposition~\ref{prop:Minsker} considers only a bound on $\lmax(\Sm_n)$.
Further, Theorem~\ref{thm:dim_free_bounded} presents a Bennett-type bound, while Proposition~\ref{prop:Minsker} presents a Bernstein-type bound. 

\paragraph*{The result by Klochkov and Zhivotovskiy, Propositon~\ref{prop:Klochkov_and_Yhivotovskiy}.}
This result presents an i.i.d.\ case with $\alpha = 1$. Theorem~\ref{thm:dim_free_unbounded} focuses on a more general setting. Specifically, we consider an arbitrary $\alpha > 0$ and let the observations be martingale differences.

\section{Corollaries}
\label{sec:corollaries}

\subsection{Straightforward corollaries} 
Theorem~\ref{thm:martingale} entails several trivial corollaries. 
For the sake of brevity, we provide them only for the case $\alpha \ge 1$.
\begin{corollary}[I.i.d.\ case]
\label{cor:iid}
    Let $\X_1, \dots, \X_n \in \H(d)$ be i.i.d.\ random matrices and 
    \[
    \E \X_1 \preccurlyeq \bm{0},\quad 
    \lmax(\E \X^2_1) \le \sg^2,\quad
    \norm*{\lmax(\X_1)_+}_{\pa} \le K,\quad
    z \eqset \lr{4 \logg \frac{e K}{\sg}}^{1/\alpha} .
    \]
    Then, with probability at least $1 - d e^{-\xx}$,
    \[
    \lmax\lr{\frac{1}{n} \sum_i \X_i} 
    \le \sg \sqrt{\frac{2 \xx}{n}} + \frac{4 K z}{\min\lrc{\azh, \logg \lr{\lr{ \frac{K z}{\sg}}^2 \frac{\xx}{n}} }} \frac{\xx}{n}.
    \]
    The same holds for \eqref{res:ber}.
\end{corollary}

\begin{proof}
    The proof is trivial. 
    The assumptions of the Corollary ensure the validity of assumptions from Theorem~\ref{thm:martingale} with $p = 0$ and $n \sg^2$, $n K^2$ instead of $\sg^2$, $U^2$.
\end{proof}

\begin{corollary}[Scalar variables]
\label{cor:1d_case}
    Let $X_1, \dots, X_n$ be scalar random variables satisfying assumptions of Theorem~\ref{thm:martingale} with
    \[
    \sg^2_i \eqset \E [X^2_i | \F_{i-1}],
    \quad
    U_i \eqset \norm*{(X_i)_{+} | \F_{i-1}}_{\pa}.
    \]
    Fix $\sg > 0$, $U \ge K > 0$ and define the event
    \begin{equation}
    \label{eq:E_scalar}
        E \eqset \lrc{\sum^n_{i = 1} \sg_i^2 \le \sg^2, \quad
        \sum^n_{i = 1} U^{2}_i \le U^2, \quad
        \max_{i \in [n]} U_i \le K}.
    \end{equation}
    Then, with probability at least $1 - \P(E) - e^{-\xx}$,
    \begin{equation}\label{res:2_corr}
        \max_{k \in [n]} \sum^k_{i=1} X_i \le \sg \sqrt{2 \xx} + \frac{4 \kz \xx}{\min\lrc{\azh,\, \logg \lr{\kzss \xx} }}.
    \end{equation}
   The same holds for \eqref{res:ber}.
\end{corollary}

\begin{corollary}[Covariance matrices]
\label{corr:covariances}
    Let $X_1, \dots, X_n \in \R(d)$ be independent random vectors, such that $U_i \eqset \norm*{\norm{X_i}}_{\pa} < \infty$ for some $\alpha \ge 2$. Denote 
    \[
    \Sg \eqset \sum_{i=1}^n \E X_i X^T_i, \quad
    \sg^2 \eqset \norm*{\Sg}, 
    \quad
    U^2 \eqset \sum_{i=1}^n U_i^2, \quad
    K \eqset \max_{i \in [n]} U_i ,\quad
    z \eqset z\lr{U, \sg; \frac{\alpha}{2}}.
    \]
    Then, with probability at least $1 - d e^{-\xx}$, it holds
    \[
    \lmax\lr{\sum^n_{i=1} X_i X^T_i - \Sg} \le 2 \sg K z \sqrt{\xx} + \frac{4 K^2 z}{\min\lrc{\azh, \logg \lr{\lr{ \frac{K }{\sqrt{2}\sg }}^2 \xx} }} \xx .
    \]
\end{corollary}

\begin{proof}
    Let $\Q_i \eqset X_i X^T_i - \E X_i X^T_i $. Note that
    \[
    \E \Q^2_i = \E \lr{X_i X^T_i}^2 - \lr{\E X_i X^T_i}^2 \preccurlyeq \E \lr{X_i X^T_i}^2.
    \]
    Denote $\Y_i \eqset \lr{X_i X^T_i}^{1/2}$. Since $\norm{\Y_i} = \norm{X_i}$, $\norm*{\norm{\Y_i}}_{\pa} = U_i$. Thus, Lemmata \ref{lemma:bound_on_e_z} and \ref{lemma:bound_on_4th_moment} ensure
    \[
    \E \Y^4_i \preccurlyeq z^2 U^2_i \lr{\E \Y^2_i + 2z^2 U^2_i e^{-z^{\alpha}}\Ii} \preccurlyeq z^2 K^2\lr{\E X_i X^T_i + \frac{2}{3} \frac{U^2_i}{U^2} \sg^2 \Ii} .
    \]
    Therefore,
    \begin{equation}
    \label{eq:bound_for_lower_bound}
        \sum_{i=1}^n \E \Q^2_i \preccurlyeq z^2 K^2 \lr{\Sg + \frac{2}{3} \sg^2 \Ii} \preccurlyeq 2 (\sg K z)^2 \Ii .
    \end{equation}
    Next, we notice that $\lmax(\Q_i) \le \lmax( X_i X^T_i) = \norm*{X_i}^2$. Consequently,
    \begin{gather*}
        \norm*{\lmax(\Q_i)_{+}}_{\psi_{\alpha/2}} \le 
        \norm*{\norm{X_i}^2}_{\psi_{\alpha/2}} = \norm*{\norm{X_i}}^2_{\psi_{\alpha}} = U^2_i,\\
        \sum_{i=}^n \norm*{\lmax(\Q_i)_{+}}_{\psi_{\alpha/2}}^2 \le \sum_{i=1}^n U_i^4 \le K^2 U^2 \le 2 (U K z)^2.
    \end{gather*}
    Applying Theorem~\ref{thm:martingale} to $\sum_{i=1}^n \Q_i$, we get the upper bound.
\end{proof}

\begin{remark}
    A lower bound on $ \sum^n_{i=1} X_i X^T_i - \Sg$ can easily be obtained using Theorem 6.1 by \cite{tropp2012user}. It is enough to notice that $-\Q_i \preccurlyeq \E X_i X^T_i$ and to use \eqref{eq:bound_for_lower_bound}. Specifically, with probability at least $1 - d e^{-\xx}$ 
    \[
    \lmax\lr{\Sg -\sum^n_{i=1} X_i X^T_i } \le 2(Kz)^2 h^{-1}\lr{\frac{\sg^2 \xx}{2(Kz)^2}} \le 2 \sg Kz \sqrt{\x} + \frac{2\sg^2 \xx}{\logg\lr{\lr{\frac{\sg}{Kz}}^2\xx}} .
    \]
\end{remark}

\subsection{Empirical Bernstein inequality}
\label{sec:emp_ber}

This section presents a modified bound \eqref{res:ber}. 
Specifically, we replace $\sg$ with its data-driven estimator $\hat{\sg}$. For the sake of simplicity, we focus on the case $\alpha \ge 1$. 
However, a similar result holds for $\alpha < 1$.

\begin{corollary}\label{cor:empirical}
    Let $\X_1, \dots, \X_n \in \H(d)$ be i.i.d.\ with
    \[
    \Sg \eqset \E (\X_1 - \E \X_1)^2, 
    \quad
    \lmax(\Sg) \le \sg^2,
    \quad
    \norm{\norm{\X_1-\E \X_1}}_{\pa} \le K.
    \]
    Denote
    \[
    \bar{\X} \eqset \frac{1}{n} \sum_i \X_i,
    \quad
    \hat{\Sg} \eqset \frac{1}{n} \sum_i (\X_i - \bar{\X})^2,
    \quad
    \hat{\sg}^2 \eqset \lmax(\hat{\Sg}),
    \]
    and define $\hat{z} \eqset z(K, \hat{\sg}; \alpha) \eqset \lr{4 \logg \frac{K e}{\hat{\sg}}}^{1/\alpha}$. Then for any $\xx > 0$ such that $n \ge 8 \xx$, 
    with probability at least $1 - 3 d e^{-\xx}$ holds
    \[
    \norm*{\bar{\X} - \E \X_1} \le \hat{\sg}\sqrt{2 \frac{\xx}{n}} + 15 K \hat{z}\frac{\xx}{n}.    
    \]
\end{corollary}

The proof is in Section~\ref{sec:proof_emp_ber}.

\subsection{McDiarmid inequality}
\label{sec:McDiarmid}

This section introduces a variant of McDiarmid's inequality, a concentration inequality that bounds how far a function of independent random variables can deviate from its expectation.
In the following, we consider $\alpha \ge 1$. Let $\mathcal{Y}$ be a measurable space and $Y_1, \dots, Y_n \in \mathcal{Y}$ be independent random variables. 
Denote $Y \eqset (Y_1, \dots, Y_n)$ and set $Y' \eqset (Y'_1, \dots, Y'_n)$ to be an i.i.d.\ copy of $Y$. 

For all $k \in [n]$ define $\sigma$-algebras
\[
\F_{-k} \eqset \sg\left(Y_1, \dots, Y_{k-1}, Y_{k+1}, \dots, Y_n\right). 
\]

Let $f \colon \mathcal{Y}^n \to \R$ be a measurable function. 
Denoting as $y = (y_1, \dots, y_n) \in \mathcal{Y}^n$ a non-random vector, we define
\[
f_i(y) \eqset f(y) - \E' f(y_1, \dots, y_{i-1}, Y'_i, y_{i+1}, \dots, y_n),
\]
where $\E'$ is the expectation w.r.t.\ $Y'$.

\begin{corollary}\label{corollary_maurer1}
    Set
    \[
    U_i \eqset \norm*{f_i(Y) \bigm| \F_{-i}}_{\pa},
    \quad
    \sg^2_i \eqset \E\lrs{f^2_i(Y) \bigm| \F_{-i}},
    \]
    and let the following inequalities hold a.s.:
    \begin{gather*}
        \max_{k} U_{k} \le K \le U,
        \quad
        \sum_{k} U_{k}^2 \le U^2 ,
        \quad
        \sum_k \sg^2_{k} \le \sg^2.
    \end{gather*}
    Then, with probability at least $1 - e^{-\xx}$,
    \begin{equation}
    \label{eq:bound_corr_ben}
        f(Y) - \E f(Y)\le \sg \sqrt{2 \xx \frac{n+1}{n}} + \frac{4 K z \xx}{\min \lrc{2 \az,\, \logg \lr{ \left(\frac{K z}{\sg \frac{n+1}{n}} \right)^2 \xx} }}.
    \end{equation}
    Moreover, with probability at least $1-e^{-\xx}$,
    \begin{equation}\label{eq:bound_corr_ber}
        f(Y) - \E f(Y) \le \sg \sqrt{2 \xx \frac{n+1}{n}} + \frac{3}{4} K z \xx \frac{n+1}{n}.
    \end{equation}
\end{corollary}

The proof is postponed to Section~\ref{sec:proof_mcdiarmid}. 
The next corollary specifies the result for the case of $\mathcal{Y}$ being a normed space.

\begin{corollary}
\label{corollary:maurer2}
    Let $(\mathcal{Y}, \norm{\cdot})$ be a normed space and assume $Y_1, \dots, Y_n$ are independent random variables. Set
    \[
    K \eqset \max_{i \in [n]} \norm*{\norm{Y_i}}_{\pa},
    \quad
    U^2 \eqset \sum_{i \in [n]} \norm*{\norm{Y_i}}^2_{\pa},
    \quad
    \sg^2 \eqset \sum_{i \in [n]} \E\norm{Y_i}^2, 
    \quad
    z \eqset \lr{4 \logg \frac{e U}{\sg}}^{1/\alpha}
    \]
    Then, with probability at least $1-e^{-\xx}$, the bounds \eqref{eq:bound_corr_ben} and \eqref{eq:bound_corr_ber} hold for $f(y) \eqset \frac{1}{2} \norm*{\sum_i y_i}$.
\end{corollary}

The proof is in Section~\ref{sec:proof_mcdiarmid}. 
For completeness, we provide the results by \citet{maurer2021concentration}. 

\begin{proposition}[Theorem~4 in \citep{maurer2021concentration}]
\label{prop:maurer_1}
    In the setting of Corollary~\ref{corollary_maurer1} it holds for $\alpha = 1$, with probability at least $1 - e^{-\xx}$, that
    \[
    f(Y) - \E f(Y) \le 2 e U \sqrt{\xx} + 2 e K \xx.
    \]
\end{proposition}

\begin{proposition}[Proposition 7 (i) in \citep{maurer2021concentration}]
    In the setting of Corollary~\ref{corollary:maurer2} it holds for $\alpha = 1$, with probability at least $1 - e^{-\xx}$, that
    \[
    \norm*{\sum_i Y_i} - \E \norm*{\sum_i Y_i} 
    \le 4 e U \sqrt{\xx }
    + 4 e K \xx.
    \]
\end{proposition}

Note that the bounds do not depend on $\sg$. 
Specifically, $U$ is used as a proxy for $\sigma$.

\section{Proofs}
\label{sec:proofs}

\subsection{Proof of Theorem~\ref{thm:martingale}}
\label{sec:proofs_main}

The proof relies on the Chernoff method for (super-)martingales introduced in \cite{freedman1975tail} and further generalized by \cite{tropp2011freedman} to the matrix case.
For the sake of completeness, let us provide it here with the proofs. 
We start with a simple generalization of Markov's inequality to supermartingales obtained by Doob \citep{Shiryaev2019prob2}. For the sake of completeness of the presentation, we provide its proof here.

\begin{proposition}\label{prop:mart_markov}
    Let $X_0, \dots, X_n$ be a non-negative supermartingale adapted to the filtration $(\F_i)_{i=0}^n$.
    Then for any $t > 0$
    \[
    \P\lrc{\max_{i \in [n]} X_i \ge t} \le \frac{\E X_0}{t} .
    \]
\end{proposition}

\begin{proof}
    Define the events
    \[
    A_k \eqset \lrc{X_k \ge t}, \quad 
    B_k \eqset \bigcup_{i=1}^{k} A_i, \quad
    C_k \eqset A_k \setminus B_{k-1} 
    \]
    and the stopping time
    \[
    \tau(\omega) \eqset n \wedge \min\lrc{k \in [n] : \omega \in A_k}, \quad \omega \in \Omega
    \]
    (with convention $\min \varnothing = \infty$).
    Since (a) $(X_i)_{i=0}^n$ is a supermartingale, (b) $C_1, \dots, C_n$ are disjoint, and (c) $X_\tau \ind[C_k] = X_k \ind[C_k] \ge t \ind[C_k]$ by the definition of $C_k$,
    \begin{align*}
        \E X_0 &\stackrel{(a)}{\ge} \E X_\tau \stackrel{(b)}{\ge} \E X_\tau \sum_{k=1}^n \ind[C_k] \stackrel{(c)}{=} \E \sum_{k=1}^n X_k \ind[C_k] \\
        &\stackrel{(c)}{\ge} \E \sum_{k=1}^n t \ind[C_k] = t \P(B_n) = t \P\lrc{\max_{i \in [n]} X_i \ge t}.
    \end{align*} 
\end{proof}

The next proposition is a master bound in the core of the proof. 
It is a version of Theorem~2.3 from \cite{tropp2011freedman}.

\begin{proposition}\label{prop:mart_bound}
    Let $(\Y_i)_{i=1}^n \subset \H(d)$ be a matrix-valued stochastic process adapted to filtration $(\F_i)_{i=1}^n$. 
    Define
    \begin{gather*}
        \V_i \eqset \ln \E\left[e^{\Y_i} \middle| \F_{i-1}\right] \in \H(d), \quad i = 1, \dots, n,\\
        \Z_k \eqset \sum_{i=1}^k (\Y_i - \V_i), \quad k = 0, \dots, n.
    \end{gather*}

    Then for any $\A \in\H(d)$ it holds that
    \begin{equation}
    \label{eq:bound_for_Maurer}
        \E\left[\tr \elr{\Z_k - \A} \middle| \F_{k-1}\right] 
        \le \tr \elr{\Z_{k-1} - \A} \text{~~a.s.\ for all~~} k \in [n],
    \end{equation}
    and
    \[
    \P\lrc{\max_{k \in [n]} \lm_{\max}\left(\Z_k - \A\right) \ge 0}
    \le \tr e^{-\A} .
    \]
\end{proposition}

\begin{proof}
    By Lieb's theorem the function $\X \mapsto \tr \elr{\Hh + \ln \X}$ is concave on $\H_{++}(d)$ for any fixed $\Hh \in \H(d)$ \cite[Theorem 6]{lieb1973convex}, thus by Jensen's inequality for all $k \in [n]$
    \begin{align*}
        \E\left[\tr \elr{\Z_k - \A} \middle| \F_{k-1}\right] &= \E\left[\tr \elr{\Z_{k-1} - \V_k + \ln e^{\Y_k} - \A} \middle| \F_{k-1}\right] \\
        &\le \tr \elr{\Z_{k-1} - \V_k + \ln \E\left[e^{\Y_k} \middle| \F_{k-1}\right] - \A} \\
        &= \tr \elr{\Z_{k-1} - \A} \text{~~a.s.}
    \end{align*}
    \cite[see][Corollary~1.5]{tropp2011freedman}.
    Then by Proposition~\ref{prop:mart_markov}
    \[
    \P\lrc{\max_{k \in [n]} \tr e^{\Z_k - \A} \ge 1 } 
    \le \E \tr e^{\Z_0 - \A} = \tr e^{-\A} .
    \]
    Finally, if $\lmax(\Z_k - \A) \ge 0$, then $\tr e^{\Z_k - \A} \ge 1$, thus the claim follows.
\end{proof}

In the next lemmata, we often use the following simple fact \citep[see][Lemma~3.1]{freedman1975tail}.

\begin{proposition}\label{prop:phi_properties}
    The function $\frac{\phi(t)}{t^2}$, extended at $0$ by continuity to $\frac{1}{2}$, is analytic and increasing on $\R$.
\end{proposition}
 
\begin{lemma}\label{lemma:rho}
    Fix $\lm > 0$, $\alpha > 0$. Define the function
    \begin{equation}\label{def:rho}
        \rho_{\lm, \alpha}(x) \eqset \left(\phi(\lm x) - \frac{(\lm x)^2}{2}\right) \elr{- x^\alpha}.
    \end{equation}
    Then for $x > 0$ it holds that
    \[
    \sgn \rho_{\lm, \alpha}'(x) = \sgn\left(\upsilon(\lm x) - \alpha x^\alpha\right) ,
    \]
    where
    \begin{equation}
        \upsilon(t) \eqset \frac{t \phi(t)}{\phi(t) - t^2 / 2} .
    \end{equation}
\end{lemma}
 
We postpone the proof to Appendix~\ref{sec:auxiliary}.

\begin{lemma}\label{lemma:phi_bound}
    Fix $\lm, \alpha > 0$. Let $y > 0$ satisfy
    \begin{equation}
    \label{eq:condition_on_y}
        \upsilon(\lm y) \le \alpha y^\alpha.
    \end{equation}
    If $\alpha < 1$, let $\tau > 0$ satisfy \eqref{eq:condition_on_y} as well, i.e., $\upsilon(\lm \tau) \le \alpha \tau^\alpha$.

    Then for all $x \in \R$ in case $\alpha \ge 1$ and for all $x \le \tau$ in case $\alpha < 1$ it holds that
    \[
    \phi(\lm x) \le x^2 \frac{\phi(\lm y)}{y^2} + \rho_{\lm, \alpha}(y) \elr{x_+^\alpha} \ind[x > y],
    \]
    where $\rho_{\lm, \alpha}(x)$ is defined by~\eqref{def:rho}.
\end{lemma}

\begin{proof}
    By the monotonicity of $\frac{\phi(t)}{t^2}$,
    \[
    \phi(\lm x) \ind[x \le y] 
    \le (\lm x)^2 \frac{\phi(\lm y)}{(\lm y)^2} \ind[x \le y] 
    = x^2 \frac{\phi(\lm y)}{y^2} \ind[x \le y] .
    \]
    
    \paragraph*{Case $\alpha \ge 1$.}
    Consider $x \ge y$. 
    The monotonicity of $\frac{\phi(t)}{t^2}$ yields that
    \[
    0 < \frac{\upsilon(\lm x)}{\lm x} = \frac{1}{1 - \frac{(\lm x)^2}{2 \phi(\lm x)}} 
    \le \frac{1}{1 - \frac{(\lm y)^2}{2 \phi(\lm y)}} = \frac{\upsilon(\lm y)}{\lm y}.
    \]
    Thus,
    \[
    \upsilon(\lm x)- \alpha x^\alpha \le \frac{x}{y} \left(\upsilon(\lm y) - \alpha y^\alpha \right) \le 0.
    \]
    Therefore, by Lemma~\ref{lemma:rho}, $\rho_{\lm, \alpha}$ is decreasing on $[y, \infty)$. 
    Since $\frac{\phi(\lm y)}{(\lm y)^2} \ge \lim_{t \to 0} \frac{\phi(t)}{t^2} = \frac{1}{2}$, we get
    \begin{align*}
        \phi(\lm x) \ind[x > y] &= \left(\phi(\lm x) - \frac{(\lm x)^2}{2}\right) \ind[x > y] + \frac{(\lm x)^2}{2} \ind[x > y] \\
        &= \rho_{\lm, \alpha}(x) \elr{x_+^\alpha} \ind[x > y] + \frac{(\lm x)^2}{2} \ind[x > y] \\
        &\le \rho_{\lm, \alpha}(y) \elr{x_+^\alpha} \ind[x > y] + x^2 \frac{\phi(\lm y)}{y^2} \ind[x > y].
    \end{align*}
    Combining the above inequalities, we obtain the first result.

    \paragraph*{Case $0 < \alpha < 1$.}
    If $y \ge \tau$, then we have a simple bound
    \[
    \phi(\lm x) \le x^2 \frac{\phi(\lm \tau)}{\tau^2} \le x^2 \frac{\phi(\lm y)}{y^2}, \quad x \le \tau.
    \]
    Now, if $0 < y < \tau$, then, due to the convexity of $\upsilon(x)$ (Lemma~\ref{lemma:convexity_xphi}) and the concavity of $x^\alpha$,
    \[
    \upsilon(\lm x) - \alpha x^\alpha \le 0, \quad y \le x \le \tau .
    \]
    Thus, $\rho_{\lm, \alpha}$ is non-increasing on $[y, \tau]$, and the second claim follows the same way as the first one.
\end{proof}
 
The next lemma ensures a bound on a matrix moment-generating function.
 
\begin{lemma}\label{lemma:mgf_X_bound}
    Let $\X \in \H(d)$ be a random matrix such that
    \[
    \E \X \preccurlyeq 0, \quad
    \E \X^2 = \Sg, \quad 
    \norm*{\lmax(\X)_+}_{\pa} = U < + \infty,
    \]
    for some $\alpha > 0$.
    
    Fix $\lm > 0$. Let $y > 0$ satisfy 
    \begin{equation}\label{eq:condition_lambda}
        \upsilon(\lm y) \le \alpha \left(\frac{y}{U}\right)^\alpha.
    \end{equation}
    If $\alpha < 1$, let $\lmax(\X) \le \tau ~~\text{a.s.}$ for some $\tau > 0$ satisfying~\eqref{eq:condition_lambda}, i.e. $\upsilon(\lm \tau) \le \alpha \left(\frac{\tau}{U}\right)^\alpha.$
    Then
    \[
    \E \elr{\lm \X} \preccurlyeq \Ii + \frac{\phi(\lm y)}{y^2} \Sg + 2 \left(\phi(\lm y) - \frac{\lm^2 y^2}{2}\right) \elr{- \left(\frac{y}{U}\right)^\alpha} \Ii.
    \]
\end{lemma}

\begin{proof}
    By rescaling, it is enough to consider the case $U = 1$.
    First, we recall that the moment-generating function satisfies
    \begin{equation*}\label{eq:phi_ineq_0} 
        \E \elr{\lm \X} = \Ii + \lm \E \X + \E\lr{e^{\lm \X} - \Ii - \lm \X} 
        \preccurlyeq \Ii + \E \phi(\lm \X).
    \end{equation*}

    Further, we can apply Lemma~\ref{lemma:phi_bound} because its conditions are fulfilled,
    \begin{align*}
        \E \phi(\lm \X) &\preccurlyeq \E \left(\frac{\phi(\lm y)}{y^2} \X^2 + \rho_{\lm, \alpha}(y) \elr{\X_+^\alpha}\right) \\
        &\preccurlyeq \frac{\phi(\lm y)}{y^2} \Sg + \rho_{\lm, \alpha}(y) \E \elr{\lmax(
        \X)_+^\alpha} \Ii \\
        &\preccurlyeq \frac{\phi(\lm y)}{y^2} \Sg + 2 \rho_{\lm, \alpha}(y) \Ii
    \end{align*}

    By replacing $y \rightarrow \frac{y}{U}$, $\Sg \rightarrow \frac{1}{U^2} \Sg$ and $\lm \rightarrow \lm U$, we get the result.
\end{proof}

\begin{lemma}
\label{lemma:bound_on_tail}
    Let $X_1, \dots, X_n$ be non-negative random variables adapted to a filtration $(\F_i)^n_{i=1}$, $\alpha > 0$, and $U_i \eqset \norm*{X_i \bigm| \F_{i-1}}_{\pa} \in [0, \infty]$.
    Fix $U \ge K > 0$ and set
    \[
    E_n \eqset \lrc{\sum^n_{i=1} U^2_i \le U^2,~\max_{i\in [n]} U_i \le K }.
    \]
    Then for any $\tau \ge K$
    \begin{equation}
    \label{eq:bound_on_p_tail}
        \P\left( \{\max_{i \in [n]} X_i \ge \tau\} \cap E_n \right) \le 2 \lr{\frac{4}{\alpha e}}^{\frac{2}{\alpha}}
        \frac{U^2}{\tau^2} e^{-\frac{1}{2} \lr{\frac{\tau}{K}}^{\alpha} }.
    \end{equation}
    Moreover, if 
    \begin{equation*}
    \label{def:bound_on_tau}
        \tau \ge K \left(2 \xx + 4 \ln\lr{\frac{2 U}{K}} + \frac{4}{\alpha} \ln \lr{\frac{4}{\alpha e}} \right)^{1/\alpha} ,
    \end{equation*}
    then
    \begin{equation}
    \label{eq:bound_tail_ex}
        \P\lr{\{\max_{i\in [n]} X_i \ge \tau\} \cap E_n} \le e^{-\xx}.
    \end{equation}
\end{lemma}

\begin{proof}
    First, we derive the key ingredient of the lemma, a bound on the indicator function $\ind[s \ge t]$ for any $t > 0$ and $s\ge 0$.
    
    Lemma~\ref{lemma:useful_bound} ensures $e^{t^{\alpha}} \ge \lr{\frac{\alpha e}{4}}^{\frac{4}{\alpha}}t^4$. Thus, for any $s \ge t > 0$
    \[
    e^{s^{\alpha}} \ge e^{t^{\alpha}} \ge \lr{\frac{\alpha e}{4}}^{\frac{2}{\alpha}}t^2e^{\frac{t^{\alpha}}{2}} ~~\Rightarrow~~ \lr{\frac{4}{\alpha e }}^{\frac{2}{\alpha}} \frac{1}{t^2} e^{-\frac{t^\alpha}{2}} e^{s^{\alpha}} \ge 1.
    \]

    This entails for any $t > 0$ and $s\ge 0$,
    \begin{equation}
    \label{eq:bound_on_ind1}
         \ind[s \ge t] \le \lr{\frac{4}{\alpha e }}^{\frac{2}{\alpha}} \frac{1}{t^2} e^{-\frac{t^\alpha}{2}} e^{s^{\alpha}}.   
    \end{equation}

    Now define auxiliary events
    \begin{gather}
    \label{def:tail_events}
        E_{k} \eqset \lrc{\sum^k_{i=1} U^2_i \le U^2,~\max_{i\in [k]} U_i \le K } \in \F_{k-1}, \quad k \in [n].
    \end{gather}
    Note that $ E_1 \supset \dots \supset E_n \supset E_{n+1} \eqset \varnothing$.

    The union bound ensures
    \[
    \P\lr{\{\max_i X_i \ge \tau\} \cap E_n} \le \sum_k \P\lr{\{ X_i \ge \tau\} \cap E_n} 
    \le \sum_i \P\lr{\{ X_i \ge \tau\} \cap E_i}.
    \]

    Now, we are to bound $\P\lr{\{ X_i \ge \tau\} \cap E_i}$. 
    In the following bound w.l.o.g., we consider all $U_i > 0$ a.s. 
    Otherwise, one could consider instead of $U_i$ the upper bound $U_i + \eps$ and let $\eps \to 0$.
    Notice that
    \begin{align}
    \label{eq:bound_tail_aux}
        &\P \lr{\{ X_i \ge \tau\} \cap E_i} = \E \ind \lrs{ X_i \ge \tau }\cdot \ind\lrs{E_i}\nonumber \\
        & \overset{{\text{by }\eqref{eq:bound_on_ind1} }}{\le} 
        \E \lr{\frac{4}{\alpha e }}^{\frac{2}{\alpha}} \lr{\frac{U_i}{\tau}}^2 e^{-\frac{1}{2}\lr{\frac{\tau}{U_i}}^{\alpha} } e^{\lr{\frac{X_i}{U_i}}^{\alpha}} \ind\lrs{E_i}\nonumber\\
        & = \E \E\lrs{ \lr{\frac{4}{\alpha e }}^{\frac{2}{\alpha}} \lr{\frac{U_i}{\tau}}^2 e^{-\frac{1}{2}\lr{\frac{\tau}{U_i}}^{\alpha} } e^{\lr{\frac{X_i}{U_i}}^{\alpha}} \ind\lrs{E_i} \middle| \F_{i-1} } \nonumber \\
        & = \E\lr{\lr{\frac{4}{\alpha e }}^{ \frac{2}{\alpha} } \lr{\frac{U_i}{\tau}}^2 e^{-\frac{1}{2}\lr{\frac{\tau}{U_i}}^{\alpha} } \ind\lrs{E_i} \cdot \E\lrs{e^{\lr{\frac{X_i}{U_i}}^{\alpha}} \middle | \F_{i-1}} } \nonumber\\
        &\le 
        \lr{\frac{4}{\alpha e}}^{ \frac{2}{\alpha} }
        \frac{2}{\tau^2} e^{-\frac{1}{2}\lr{\frac{\tau}{K}}^{\alpha} } \E U^2_i \ind\lrs{E_i}.
    \end{align}
    The last inequality holds because $U_i \le K$ on $E_i$ and $\E\lrs{\elr{\lr{\frac{X_i}{U_i}}^{\alpha}} \bigg| \F_{i-1}} \le 2$.

    Now, we consider
    \begin{align}
    \label{eq:bound_sum_uk}
        \sum_k \E U^2_k \ind\lrs{E_k} &= \E \sum^n_{k=1} U^2_k \sum^n_{i=k}\lr{\ind\lrs{E_i} - \ind\lrs{E_{i+1}}} \nonumber \\
        &=\E \sum^n_{i=1}\lrs{\lr{\ind\lrs{E_i} - \ind\lrs{E_{i+1}}}\sum^i_{k=1} U^2_k } \nonumber \\
        & \overset{(a)}{\le} U^2\E\sum^n_{i=1}\lr{\ind\lrs{E_i} - \ind\lrs{E_{i+1}} } = U^2\E\ind\lrs{E_1} \le U^2,
    \end{align}
    where (a) holds because on any event $E_k$ it holds $\sum^k_{i=1}U^2_i \le U^2$.
    Combining \eqref{eq:bound_tail_aux} and \eqref{eq:bound_sum_uk}, we get the first result \eqref{eq:bound_on_p_tail}. 

    To get \eqref{eq:bound_tail_ex}, one has to find $\tau$ s.t.
    \[
    -\xx \ge \frac{2}{\alpha}\ln \lr{\frac{4}{\alpha e}} + \ln 2 + 2 \ln\lr{\frac{U}{\tau}} -
    \frac{1}{2}\lr{\frac{\tau}{K}}^{\alpha}.
    \]
    Thus,
    \[
    \lr{\frac{\tau}{K}}^{\alpha} - 4\ln\lr{\frac{U}{\tau}} \ge 2 \xx + \ln 4 + \frac{4}{\alpha} \ln \lr{\frac{4}{\alpha e}} 
    \]
    This is equivalent to
    \[
    \lr{\frac{\tau}{K}}^{\alpha} + 4\ln\lr{\frac{\tau}{K}} - 4 \ln\lr{\frac{U}{K}} \ge 2 \xx + \ln 4 + \frac{4}{\alpha}\ln \lr{\frac{4}{\alpha e}}.
    \]
    As $\ln 4 > 1$, $U \ge K$, and $\tau \ge K$, the second claim follows.
\end{proof}

Now, we are ready to prove the main result.
 
\begin{proof}[Proof of Theorem~\ref{thm:martingale}]
    Define
    \begin{equation}
    \label{def:ok_tau}
        \tau \eqset 
        \begin{cases}
            K \max\lrc{z,\; \left(2 \xx + 4 \ln\lr{\frac{2 U}{K}} + \frac{4}{\alpha}\ln \lr{\frac{4}{\alpha e}} \right)^{1/\alpha} }, & \alpha < 1,\\
            K z, & \alpha \ge 1.
        \end{cases}
    \end{equation}

    Set $Y_i \eqset U_i z \le K z$ on $E$. Let
    \begin{equation}
    \label{def:lm_0}
        \lm_0 \eqset \frac{2 \alpha }{3 K} \lr{\frac{\tau}{K}}^{\alpha - 1} \le \frac{2 \alpha}{3 K} z^{\alpha - 1},
    \end{equation}
    and consider $0 \le \lm \le \lm_0$. 
    Then by Lemma~\ref{lemma:convexity_xphi}
    \[
    \upsilon(\lm Y_i) \le \upsilon(\lm K z) < \min\{4, 1.5 \lm K z\} \le \alpha z^\alpha = \alpha \left(\frac{Y_i}{U_i}\right)^\alpha ~~\text{on}~~E.
    \]

    Now, we set
    \[
    \tilde{\X}_i \eqset 
    \begin{cases}
        \X_i, &~~\text{if}~~ \alpha \ge 1,\\
        \X_i \ind_{(-\infty, \tau]}(\X_i), & \text{otherwise} .
    \end{cases}
    \]
    By construction, for all $i \in [n]$ one has $\tilde{\X}_i \preccurlyeq \X_i$ and $\tilde{\X}_i^2 \preccurlyeq \X_i^2$, thus
    \[
    \E \tilde{\X}^2_i \preccurlyeq \E \X^2_i, 
    \quad
    \norm*{\lmax(\tilde{\X}_i)_{+}}_{\pa} \le U_i.
    \]
    Moreover, if $\alpha < 1$, then $\lmax(\tilde{\X}_i)_{+} \le \tau$ by construction.
    
    Denote 
    \[
    \tilde{\Sm}_k = \sum^k_{i=1} \tilde{\X}_i.
    \]

    By Lemma~\ref{lemma:mgf_X_bound}, using the monotonicity of $\frac{\phi(t)}{t^2}$ and the fact that $\ln(\X)$ is a monotone map on the cone of positive-definite matrices (see (2.8) in \cite{tropp2012user}), we obtain
    \begin{align*}
        \V_i(\lm) &\eqset \ln \E \left[\elr{\lm \tilde{\X}_i} \middle| \F_{i-1}\right] \\
        &\preccurlyeq \frac{\phi(\lm Y_i)}{Y_i^2} \Sg_i + 2 \left(\phi(\lm Y_i) - \frac{\lm^2 Y_i^2}{2}\right) \elr{- \left(\frac{Y_i}{U_i}\right)^\alpha} \Ii \\
        &= \frac{\phi(\lm Y_i)}{Y_i^2} \Sg_i + 2 U_i^2 z^2 \left(\frac{\phi(\lm Y_i)}{Y_i^2} - \frac{\lm^2}{2}\right) \elr{- z^\alpha} \Ii \\
        &\preccurlyeq  \frac{\phi(\lm K z)}{(K z)^2} \Sg_i + 2 \frac{U_i^2}{K^2} \left(\phi(\lm K z) - \frac{(\lm K z)^2}{2}\right) \elr{- z^\alpha} \Ii ~~\text{on}~~E
    \end{align*}
    The last inequality is due to $Y_i = U_i z \le K z$ on $E$.

    Lemma~\ref{lemma:bound_on_e_z} ensures for all $\elr{z^\alpha} \ge \frac{e^4}{16} \left(\frac{U z}{\sg}\right)^2$.
    Therefore, for all $k \in [n]$
    \begin{align}
    \label{eq:sum_Vi}
        \sum_{i=1}^k \V_i(\lm) &\preccurlyeq \Sg \frac{\phi(\lm K z)}{(K z)^2} + 2 \frac{U^2}{K^2} \left(\phi(\lm K z) - \frac{(\lm K z)^2}{2}\right) \elr{- z^\alpha} \Ii \nonumber \\
        &\preccurlyeq \left(\frac{\sg}{K z}\right)^2 \left(\phi(\lm K z) + \frac{2}{3} \left(\phi(\lm K z) - \frac{(\lm K z)^2}{2}\right)\right) \Ii ~~\text{on}~~E.
    \end{align}
    Note that for any $a > 0$
    \[
    \phi(t) + a \left(\phi(t) - \frac{t^2}{2}\right) = \frac{t^2}{2} + (1 + a) \sum_{k=3}^\infty \frac{t^k}{k!} \le \frac{1}{(1 + a)^2} \sum_{k=2}^\infty \frac{((1 + a) t)^k}{k!} = \frac{\phi((1 + a) t)}{(1 + a)^2},
    \]
    thus for all $k \in [n]$
    \begin{equation}
    \label{eq:bound_aux_Vi}
        \sum_{i=1}^k \V_i(\lm) \preccurlyeq \left(\frac{\sg}{M}\right)^2 \phi(\lm M) \Ii ~~\text{on}~~E,~~\text{where}~~ M \eqset \frac{5}{3} K z .
    \end{equation}
    Proposition~\ref{prop:mart_bound} with $\Y_i = \lm \tilde{\X}_i$ and $\A = \left(\lm t - \left(\frac{\sg}{M}\right)^2 \phi(\lm M)\right) \Ii$
    yields that
    \begin{align*}
        \P\lrc{\max_{k \in [n]} \lmax\left(\lm \tilde{\Sm}_k - \sum_{i=1}^k \V_i(\lm) - \A\right) \ge 0} &\le \tr \elr{-\A} \\
        &= d \elr{\left(\frac{\sg}{M}\right)^2 \phi(\lm M) - \lm t} .
    \end{align*}
    Note that 
    \[
    \lmax(\lm \tilde{\Sm}_k) \le \lmax\left(\lm  \tilde{\Sm}_k - \sum_{i=1}^k \V_i(\lm) - \A\right) + \lmax\lr{\sum_{i=1}^k \V_i(\lm) + \A}.
    \]
    Moreover, \eqref{eq:bound_aux_Vi} yields on the event $E$ 
    \[
    \max_{k \in [n]} \lmax\left(\sum_{i=1}^k \V_i(\lm) + \A\right) \le \lm t,
    \]
    thus
    \begin{equation}
    \label{eq:bound_prob_Ep}
        \P\lr{\lrc{\max_{k \in [n]} \lmax(\tilde{\Sm}_k) \ge t } \cap E}
        \le d \elr{\left(\frac{\sg}{M}\right)^2 \phi\left(\lm M\right) - \lm t} .
    \end{equation}

    If $\alpha \ge 1$, $\tilde{\X}_i = \X_i$, and we immediately get $\tilde{\Sm}_k = \Sm_k$. Thus,
    \[
    \P\lrc{\max_{k \in [n]} \lmax(\Sm_k) \ge t } 
    \le 1 - \P\{E\} + d \elr{\left(\frac{\sg}{M}\right)^2 \phi\left(\lm M\right) - \lm t}.
    \]

    Consider $0 < \alpha < 1$. To get the bound on $\Sm_k$, one has to estimate the probability that $\tilde{\X}_i = \X_i$ for all $i$ and the event $E$ holds. 
    Note that, by construction, $\tilde{\X}_i \neq \X_i$ iff $\lmax(\X_i) > \tau$.
    Recall that $E \subset E_n$ with $E_n$ coming from \eqref{def:tail_events}. Thus, due to the choice of $\tau$ \eqref{def:ok_tau}, by Lemma~\ref{lemma:bound_on_tail}
    \begin{align*}
        \P\lr{\{\exists k \in [n]: \tilde{\Sm}_k \neq \Sm_k\} \cap E}
        &= \P\lr{\{\exists i \in [n]: \tilde{\X}_i \neq \X_i\} \cap E} \\
        &= \P\left(\{\exists i \in [n]: \lmax(\X_i) > \tau\} \cap E \right) \\
        &\le \P\left( \{\max_{i \in [n]} \lmax(\X_i)_+ > \tau\} \cap E_n \right) \le e^{-\xx} .
    \end{align*}

    Combining this bound with \eqref{eq:bound_prob_Ep}, we get for $\alpha < 1$
    \[
    \P\lrc{\max_{k \in [n]} \lmax(\Sm_k) \ge t }
    \le 1 - \P\{E\} + e^{-\xx} + d \elr{\left(\frac{\sg}{M}\right)^2 \phi\left(\lm M\right) - \lm t}
    \]

    \paragraph*{Optimization over $\lm$.}
    We have to minimize $\elr{\left(\frac{\sg}{M}\right)^2 \phi\left(\lm M\right) - \lm t}$ w.r.t. $\lm$. 
    Let $\xi_0 \eqset \lm_0 M$, then
    \[
    \min_{0 \le \lm \le \lm_0} \left(\frac{\sg}{M}\right)^2 \phi\left(\lm M\right) - \lm t 
    = \left(\frac{\sg}{M}\right)^2 \min_{0 \le \xi \le \xi_0} \phi(\xi) - \left(\frac{M}{\sg}\right)^2 \frac{\xi}{M} t 
    = - \left(\frac{\sg}{M}\right)^2 g_{\xi_0}\left(\frac{M t}{\sg^2}\right),
    \]
    with $g_{\xi_0}$ coming from Lemma \ref{lemma:bound_on_g_inv}.
    
    According to the previous bounds, it is enough to find $t = t(\xx)$, s.t.
    \[
    \left(\frac{\sg}{M}\right)^2 g_{\xi_0}\left(\frac{M t(\xx)}{\sg^2}\right) \ge \xx,
    \]
    that is equivalent to
    \[
    t(\xx) \ge \frac{\sg^2}{M} g_{\xi_0}^{-1}\left(\frac{M^2 \xx}{\sg^2}\right) .
    \]

    In view of Lemma \ref{lemma:bound_on_g_inv}, we choose
    \[
    t(\xx) \eqset 
    \begin{cases}
        \frac{\sg^2}{M} h^{-1}\left(\frac{M^2 \xx}{\sg^2}\right), &~\text{if}~ \xx \le \xx_0,\\
         \frac{2}{\lm_0} \xx, &~\text{if}~ \xx > \xx_0,
    \end{cases}
    \]
    with $\xx_0 \eqset \xi_0 \phi'(\xi_0) - \phi(\xi_0)$ (recall that $\xi_0 \eqset \lm_0 M$).
    
    Substituting $\lm_0$ defined in \eqref{def:lm_0}, we get for $\xx > \xx_0$
    \begin{equation}
    \label{eq:b1.1}
        t(\xx) = \frac{3 K}{\alpha} \lr{\frac{\tau}{K}}^{1-\alpha} \xx 
        \le \frac{3 K z}{\az}\xx + \frac{3 K}{\alpha}\xx\left(2 \xx + 4 \ln\lr{\frac{2 U}{K}}
        + \frac{4}{\alpha}\ln \lr{\frac{4}{\alpha e}} \right)^{\frac{1-\alpha}{\alpha}} \ind[\alpha < 1]. 
    \end{equation}

    If $\xx\le \xx_0$, then the bound \eqref{eq:bound_h_inv} on $h^{-1}(\cdot)$ yields
    \begin{align*}
    \label{eq:delta_2_bound}
        t(\xx) & \le \frac{\sg^2}{M}\lr{ \sqrt{2 \lr{\frac{M}{\sg}}^2 \xx} +  \frac{2\lr{\frac{M}{\sg}}^2\xx}{\logg \lr{2\lr{\frac{M}{\sg}}^2 \xx} } } = \sigma  \sqrt{2\xx} + \frac{2M\xx}{\logg \lr{2\lr{\frac{M}{\sg}}^2 \xx} } \nonumber \\
        & = \sg\sqrt{2 \xx} + \frac{\frac{10}{3} K z \xx }{ \logg \lr{2\lr{\frac{5 Kz}{3 \sg}}^2 \xx} }
        \le \sg\sqrt{2\xx} + \frac{4 K z \xx }{ \logg \lr{\lr{\frac{Kz}{\sg}}^2 \xx} },
    \end{align*}
    where we used that $M \eqset 5/3 Kz$.
    We get \eqref{res:ben} by combining this bound with~\eqref{eq:b1.1}.
    
    Finally, let us prove \eqref{res:ber}. 
    First, we consider the case $\xx \le \xx_0$ and apply the well-known bound on $h^{-1}(\cdot)$,
    \[
    h^{-1}(u) \le \sqrt{2u} + \frac{u}{3}.
    \]
    This yields for $\xx \le \xx_0$
    \[
    t(\xx) \le \sg \sqrt{2 \xx} + \frac{M}{3} \xx = \sg \sqrt{2 \xx} + \frac{5 K z}{9} \xx.
    \]
    
    Combining this bound with~\eqref{eq:b1.1} for the case $\xx > \xx_0$ and using the inequality $\az \ge 4$, we get~\eqref{res:ber}.
\end{proof}

\subsection{Proofs of Theorem~\ref{thm:dim_free_bounded} and~\ref{thm:dim_free_unbounded}}
\label{sec:proofs_dim_free}

To get the bounds incorporating intrinsic dimension, one has to prove a modified version of Proposition~\ref{prop:mart_bound}.

\begin{lemma}
    \label{lemma:mart_bound_dim_free}
    Let $\Y_i$, $\V_i$ and $\Z_k$ be as in Proposition~\ref{prop:mart_bound}, i.e., $(\Y_i)_{i=1}^n \subset \H(d)$ is a matrix-valued stochastic process adapted to filtration $(\F_i)_{i=1}^n$, 
    \begin{gather*}
        \V_i \eqset \ln \E\left[e^{\Y_i} \middle| \F_{i-1}\right] \in \H(d),
        \quad 
        \tilde{\V}_i \eqset \V_i - \E[\Y_i | \F_{i-1}],
        \quad i \in [n],\\
        \Z_k \eqset \sum_{i=1}^k (\Y_i - \V_i), \quad k = 0, \dots, n.
    \end{gather*}
    Consider nested events $E_1 \supset \dots \supset E_n$, such that $E_k \in \F_{k-1}$. 
    Then for any $\A \in \H(d)$ and $\eps > 0$, it holds
    \begin{align*}
        \P\lr{E_n \bigcap \left\{\max_{k\in [n]}\lmax(\Z_k + \A) \ge \eps\right\}} &\le 
        \P\lr{\bigcup^n_{k=1} E_k \bigcap \left\{\lmax(\Z_k + \A) \ge \eps\right\} } \\
        &\le \frac{\tr\Bigl(\phi(\A) + \sum_{k=1}^n \E \ind\lrs{E_k}\cdot(\tilde{\V}_k)_{+}\Bigr)}{\phi(\eps)}.
    \end{align*}
\end{lemma}

\begin{proof}
    First, we notice that Proposition~\ref{prop:mart_bound} ensures $\E \lrs{\tr e^{\Z_k + \A}\big| \F_{k-1}} \le \tr e^{\Z_{k-1} + \A}$. Consequently,
    \begin{align*}
        \E\lrs{\tr \phi\lr{\Z_k + \A} \middle| \F_{k-1}} &= \E \lrs{ \tr\lr{e^{\Z_k + \A} - \Z_k - \A - \bm{I} }\big| \F_{k-1}}\\
        &\le  \tr\lr{e^{\Z_{k-1} + \A} - \Z_{k-1} - \A - \bm{I} } - 
        \E\lrs{\tr\lr{\Y_k - \V_k}\middle| \F_{k-1}}\\
        &= \tr\phi\lr{\Z_{k-1} + \A} + \tr\lr{\V_k - \E\lrs{\Y_k\middle| \F_{k-1}}} = \tr\phi\lr{\Z_{k-1} + \A} +\tr\tilde{\V}_k.
    \end{align*}
    This yields
    \begin{align*}
        \E\lrs{\ind\lrs{E_k} \tr \phi\lr{\Z_k + \A} \middle| \F_{k-1} } = \ind\lrs{E_k} 
        \tr\lr{\phi\lr{\Z_{k-1} + \A} + \tilde{\V}_k } \le \ind\lrs{E_{k-1}}\tr\phi\lr{\Z_{k-1} + \A} + \ind\lrs{E_k} \tr(\tilde{V}_k)_{+}.
    \end{align*}
    The last inequality holds because $\ind\lrs{E_k} \le \ind\lrs{E_{k-1}}$ and $\phi \ge 0$. Here we assume $E_0 \eqset \Omega$.
    Define
    \[
    \xi_k \eqset \ind\lrs{E_{k}} \tr \phi\lr{\Z_{k} + \A} + \sum^n_{i=k+1}\tr \E\lrs{\ind\lrs{E_i} (\tilde{V}_i)_{+} \middle| \F_{k} } \ge 0.
    \]
    Further, $\ind\lrs{E_k}\tr(\tilde{\V}_k)_{+} = \E\lrs{\ind\lrs{E_k}\tr(\tilde{\V}_k)_{+}|\F_{k-1}}$ a.s., thus
    \[
    \E\lrs{\xi_k \middle| \F_{k-1} } \le \ind\lrs{E_{k-1}} \tr \phi\lr{\Z_{k-1} + \A} + \ind\lrs{E_k} \tr(\tilde{\V}_k)_{+} + \sum^n_{i=k+1} \tr \E\lrs{ \ind\lrs{E_i} \cdot (\tilde{\V}_i)_{+} \middle| \F_{k-1}} = \xi_{k+1}.
    \]
    By Proposition~\ref{prop:mart_markov}, this entails for any $t > 0$
    \begin{equation}
    \label{eq:bound_on_xi_aux}
        \P\lrc{\max_{k \in [n]} \xi_k \ge t} \le 
        \frac{\xi_0}{t} 
        =
        \frac{\tr \phi(\A) + \tr \lr{\sum^{n}_{i=1} \E \ind\lrs{E_i}\cdot(\tilde{\V}_i)_{+} }}{t}.
    \end{equation}
    Therefore,
    \begin{align*}
        &\P\lr{\bigcup_k \lrc{E_k \bigcap\lrc{\lmax(\Z_k + \A) \ge \eps} } } = \P\lr{\max_{k\in [n]} \ind\lrs{E_k} \lmax(\Z_k + A) \ge \eps}\\
        &\le \P\lrc{\max_{k\in [n]} \ind\lrs{E_k}\tr \phi(\Z_k + \A) \ge \phi(\eps)} \le \P\lrc{\max_{k\in [n]}\xi_k \ge \phi(\eps) }.
    \end{align*}
    We get the result of the lemma by combining this inequality with \eqref{eq:bound_on_xi_aux}.
\end{proof}

\begin{proof}[Proof of Theorem~\ref{thm:dim_free_bounded}]
    Denote
    \[
    E_k \eqset \lrc{\max_{i\in[k]}\norm*{\lmax(\X_i)| \F_{i-1} }_{\infty}\le K, \quad \sum^k_{i=1} \Sg_i \preccurlyeq \Sg} \in \F_{k-1}.
    \]
    Note that $E_n \subset E_{n-1} \subset \dots \subset E_1$ and by Lemma 3.2 in \citep{tropp2011freedman} 
    \[
    \V_k \eqset \ln \E\left[e^{\lm \X_k} \middle| \F_{k-1}\right] \preccurlyeq \Sg_k \frac{\phi(\lm K)}{K^2}~~~~\text{on}~E_k.
    \]
    Thus,
    \[
    \sum^n_{k=1} \ind\lrs{E_k} \cdot (\V_k)_{+} 
    \preccurlyeq
    \sum^n_{k=1} \ind\lrs{E_k} \Sg_k \frac{\phi(\lm K)}{K^2} \preccurlyeq \Sg  \frac{\phi(\lm K)}{K^2}.
    \]
    Set $\A \eqset \frac{\phi(\lm K)}{K^2} \Sg $ and note that on $E_n$ it holds
    \[
    \Z_k + \A = \lm \Sm_k - \sum^k_{i=1}\V_i + \A \succcurlyeq \lm \Sm_k.
    \]
    By Lemma~\ref{lemma:mart_bound_dim_free}
    \begin{align*}
        \P\lr{E_n \bigcap\lrc{\max_{k \in [n]} \lmax(\Sm_k) \ge t}} &\le \P\lr{E_n \bigcap\lrc{\max_{k \in [n]} \lmax(\Z_k + \A) \ge \lm t}} \\
        &\le \frac{\tr\lr{\phi(\A) + \A}}{\phi(\lm t)} = \frac{\tr\lr{e^{\A} - \bm{I}}}{\phi(\lm t)}.
    \end{align*}
    Lemma~7.5.1 by \cite{tropp2012user_nips} ensures
    \begin{align}
        \label{eq:bound_exp_phi_dim_free}
        \tr\lr{e^{\A} - \bm{I}} \le r(\A) \lr{e^{\lmax(\A)} - 1} = r(\Sg) \lr{\elr{\sg^2 \frac{\phi(\lm K)}{K^2} }- 1}.
    \end{align}

    If $\lm t \ge \max\lrc{1, \sg^2 \frac{\phi(\lm K)}{K^2}}$, it holds
    \begin{align*}
       \frac{\elr{\sg^2 \frac{\phi(\lm K)}{K^2}} - 1}{\phi(\lm t)} \overset{(a)}{\le} \frac{e-1}{e-2} \frac{\elr{\sg^2 \frac{\phi(\lm K)}{K^2} }- 1}{e^{\lm t} - 1} \overset{(b)}{\le} \frac{e-1}{e-2} \frac{\elr{\sg^2 \frac{\phi(\lm K)}{K^2} }}{e^{\lm t}} \le e \cdot \elr{\sg^2 \frac{\phi(\lm K)}{K^2} - \lm t};
    \end{align*}
    (a) holds because the function
    \[
    \frac{\phi(x)}{e^{x} - 1} = \frac{e^{x} - 1 - x}{e^{x} - 1} = 1 - \frac{x}{e^{x} -1}
    \]
    is increasing; 
    (b) holds because if $1 < a \le b$, then $\frac{a-1}{b-1} \le \frac{a}{b}$.
    
    Optimization over $\lm$ ensures 
    \[
    \min_{\lm > 0} \lr{\frac{\sg^2}{K^2} \phi(\lm K) - \lm t } = - \frac{\sg^2}{K^2} h\lr{\frac{K t}{\sg^2}} \le 0,
    \]
    where the minimum is attained at $\lm_{opt} = \frac{1}{K} \ln(1 + \frac{t K}{\sg^2})$. In particular, $\lm_{opt} t \ge \sg^2 \frac{\phi(\lm K)}{K^2}$.
    Thus, if $\lm_{opt} t \ge 1$, then the claim follows.
    Otherwise, $e \cdot \elr{\sg^2 \frac{\phi(\lm_{opt} K)}{K^2} - \lm_{opt} t} \ge 1$, and the result follows immediately, since the probability is bounded by $1$.
\end{proof}

Next, we prove the result for the unbounded case.

\begin{proof}[Proof of Theorem \ref{thm:dim_free_unbounded}]
    Fix some $y \ge \ln\lr{\frac{8}{\eps}}^{1/\alpha} K z$ and set
    \begin{gather*}
        \Y_i \eqset \X_i\ind_{(-\infty, y]}(\X_i), 
        \quad
        \bar{\Y}_i \eqset \Y_i - \E\lrs{\Y_i\middle| \F_{i-1}},
        \\
         \Z_i \eqset \X_i \ind_{(y, \infty)}(\X_i),
        \quad
        \bar{\Z}_i \eqset \Z_i - \E\lrs{\Z_i\middle| \F_{i-1}}.
    \end{gather*}

    Essentially, the proof consists of two steps. First, we bound the term $\max_{k\in [n]}\lmax(\sum^k_{i=1} \bar{\Y}_i)$, and then we bound $\max_{k\in [n]}\lmax(\sum^k_{i=1} \bar{\Z}_i)$.

    \paragraph*{A bound on $\max_{k\in [n]}\lmax(\sum^k_{i=1} \bar{\Y}_i)$.} 
    Since $\X_i$ is a martingale difference, it holds $\E\lrs{\X_i|\F_{i-1}} = 0$. This ensures $\bar{\Y}_i = \Y_i + \E\lrs{\Z_i|\F_{i-1}}$. Note that on the event $E$ it holds
    \begin{align}
    \label{eq:bound_on_Z_dim_free}
        \lmax\lr{\E\lrs{\Z_i\middle| \F_{i-1}} } 
        &\le \E\lrs{\norm{\Z_i} \middle |\F_{i-1}} \overset{\text{(a)}}{\le} \E\Bigl[\lmax(\X_i) \cdot\ind\lrs{\lmax(\X_i) > y} \Big |\F_{i-1} \Bigr]\nonumber \\ & \overset{\text{(b)}}{\le} \E \lrs{y e^{-\lr{\frac{y}{U_i}}^{\alpha} } e^{\lr{\frac{\lmax(\X_i)_+}{U_i}}^{\alpha}} \middle|\F_{i-1} } 
         \le 2 y e^{-\lr{\frac{y}{U_i}}^{\alpha} } \overset{\text{(c)}}{\le} 2^{-11}y,
    \end{align} 
    where (a) holds because $\norm{\Z_i} = \lmax(\Z_i) = \lmax(\X_i)\ind\lrs{\lmax(\X_i) > y_i}$; (b) holds by Lemma~\ref{lemma:useful_bound} since $y\ge Kz$, i.e., $\alpha \lr{\frac{y}{U_i}}^{\alpha} \ge \alpha \lr{\frac{y}{K}}^{\alpha} \ge 1$; (c) holds because $z^{\alpha} \ge 4$ and $\lr{\frac{y}{K}}^{\alpha} \ge z^\alpha \ln \frac{8}{\eps} \ge 4 \ln 8$.

    Consequently, we get
    \[
    \norm*{\lmax(\bar{\Y}_i)_{+}\middle| \F_{i-1}}_{\infty} \le y + \lmax\lr{\E\lrs{\Z_i\middle| \F_{i-1}} } \le (1 + 2^{-11})y.
    \]
    Moreover, by construction
    \[
    \E\lrs{\bar{\Y}^2_i \middle|\F_{i-1}} \preccurlyeq \E\lrs{\Y^2_i \middle|\F_{i-1}} \preccurlyeq \Sg_i.
    \]
    Now, define an event
    \begin{equation}\label{eq:aux_bounded_dim_free}
        E_Y \eqset \lrc{\max_{k\in [n]}\lmax\left(\sum^k_{i=1} \bar{\Y}_i \right) 
        \ge \sg\sqrt{2\xx} + 2 \lr{1 + 2^{-11}}y \frac{\xx}{\logg\lr{\lr{\frac{y}{\sg} }^2\xx} }}.
    \end{equation}
    Thus, applying Theorem \ref{thm:dim_free_bounded} we get that $\P(E \cap E_Y) \le e r(\Sg) e^{-\xx}$.

    \paragraph*{A bound on $\max_{k\in [n]}\lmax(\sum^k_{i=1} \bar{\Z}_i)$.} 
    Note that by construction $\Z_i \succcurlyeq 0$. Set
    \[
    \xi_i \eqset \lmax(\Z_i) = \lmax(\X_i) \cdot \ind\lrs{\lmax(\X_i) \ge y} \ge 0,
    \]
    then $\norm{\xi_i | \F_{i-1}}_{\pa} \le \norm{\lmax(\X_i)_{+}| \F_{i-1}}_{\pa} = U_i$.

    Next, we fix $y \eqset \ln\lr{\frac{8}{\eps}}^{1/\alpha} K z + \frac{\eps \sg}{2e\sqrt{\xx}}$ and apply Lemma~\ref{lemma:useful_bound} to \eqref{eq:bound_on_Z_dim_free}. It holds on the event $E$, that 
    \begin{equation}
    \label{eq:boud_boud}
        \E \lrs{\xi_i \middle |\F_{i-1}} \le 2y e^{-\lr{\frac{y}{U_i}}^{\alpha}} \le 2\frac{U^2_i}{y} \lr{\frac{4}{\alpha e}}^{\frac{2}{\alpha}}e^{-\frac{1}{2} \lr{\frac{y}{U_i}}^{\alpha}} \overset{(a)}{\le} \frac{\eps^2}{2 e^2} \frac{U^2_i}{U^2}\cdot\frac{\sg^2}{ y} ,
    \end{equation}
    where (a) holds because 
    \[
    e^{\frac{1}{2} \lr{\frac{y}{U_i}}^{\alpha}} \ge e^{\frac{1}{2} \ln\lr{\frac{8}{\eps}} z^\alpha} \ge e^{\frac{1}{2} z^\alpha + 2 \lr{\ln\lr{\frac{8}{\eps}} - 1}} \ge \frac{64}{e^2 \eps^2} \frac{U^2}{\sg^2} \lr{\frac{e}{\min\lrc{1, \alpha}}}^{\frac{2}{\min\lrc{1, \alpha}}},
    \]
    and  
    \[
    \lr{\frac{4}{\alpha e}}^{\frac{2}{\alpha}} 
    \lr{\frac{\min\{1, \alpha\}}{e}}^{\frac{2}{\min\lrc{1, \alpha}}} =
    \begin{cases}
       \lr{\frac{4}{e^2}}^{\frac{2}{\alpha}} \le \lr{\frac{4}{e^2}}^2, & \alpha \le 1\\
       \lr{\frac{4}{\alpha e}}^{\frac{2}{\alpha}}\frac{1}{e^2} < \lr{\frac{4}{e^2}}^2 , & \alpha > 1.
    \end{cases}
    \]

    Further, \eqref{eq:boud_boud} and the choice of $y$ ensure
    \begin{equation}
    \label{eq:bound_useful_exp_dim_free}  
        \sum^n_{i=1} \E \lrs{\xi_i \middle | \F_{i-1}} \le \frac{\eps^2}{2e^2} \frac{\sg^2}{y} \le \frac{\eps}{e} \sg \sqrt{\xx}.
    \end{equation}
    Similarly, we get
    \[
    \E \lrs{\xi^2_i\middle | \F_{i-1}} = \E \lrs{\lmax^2(\X_i) \ind \lrs{\lmax(\X_i) > y} \middle | \F_{i-1}} \le 2y^2 e^{-\lr{\frac{y}{U_i}}^{\alpha}} \le \frac{\eps^2}{2e^2} \frac{U^2_i}{U^2} \sg^2,
    \]
    thus
    \[
    \sum_{i=1}^n \E \lrs{\xi^2_i\middle | \F_{i-1}} \le \frac{\eps^2}{2e^2} \sg^2 .
    \]
    
    For the moment we set
    \begin{align*}
        T &\eqset \frac{3K}{\alpha}\left(2 \xx + 2\ln\lr{\frac{4 U}{K}}
        + \frac{4}{\alpha}\ln \lr{\frac{4}{\alpha e}} \right)^{\frac{1-\alpha}{\alpha}} \ind[\alpha < 1],
    \end{align*}
    and define the event
    \begin{align}
    \label{eq:yet_another_bounde_dim_free}
        E_\xi \eqset \lrc{  \max_{k \in [n]} \sum^k_{i=1}\lr{\xi_i - \E \lrs{\xi_i \middle | \F_{i-1}}} > \frac{\eps}{e} \sg \sqrt{\xx} + \frac{4 \kz \xx}{\min\lrc{2 \az, \logg\lr{\lr{\frac{Kz}{\sg}}^2 \xx} }} + T }.
    \end{align}
    Applying Theorem~\ref{thm:martingale}, we get $\P(E \cap E_\xi) \le e^{-\x} + \ind[\alpha < 1]e^{-\x}$. 
    
    \paragraph*{Mixed bound.} 
    We note that
    \[
    \lmax\lr{\sum^k_{i=1} \bar{\Z}_i} \le \sum^k_{i=1}\lmax\lr{\bar{\Z}_i} \le 
    \sum^k_{i=1}\lmax(\Z_i) = \sum^k_{i=1} \xi_i. 
    \]
    This yields
    \begin{align*}
        \lmax\lr{\Sm_k} \le \lmax\lr{\sum^k_{i=1} \bar{\Y}_i} + \lmax\lr{\sum^k_{i=1} \bar{\Z}_i} \le \lmax\lr{\sum^k_{i=1} \bar{\Y}_i} + \sum^k_{i=1}\lr{\xi_i - \E \lrs{\xi_i \middle | \F_{i-1}}} + \sum^k_{i=1}\E \lrs{\xi_i\middle | \F_{i-1}}. 
    \end{align*}
    Combining \eqref{eq:aux_bounded_dim_free}, \eqref{eq:bound_useful_exp_dim_free}, and \eqref{eq:yet_another_bounde_dim_free}, we get on the event $E \cap E_Y \cap E_\xi$ 
    \begin{align*}
        \lmax\lr{\Sm_k} & \le \sg \sqrt{2 \xx} + 2 \lr{1 + 2^{-11}} y \frac{\xx}{\logg\lr{\lr{\frac{y}{\sg}}^2\xx}} + \frac{\eps}{e} \sg \sqrt{\xx} + \frac{\eps}{e} \sg \sqrt{\xx} + \frac{4 \kz \xx}{\min\lrc{2 \az, \logg\lr{\lr{\frac{Kz}{\sg}}^2 \xx} }} + T \\
        &\le \sg \sqrt{2 \xx} + \lr{\frac{1 + 2^{-11}}{e} + \frac{2}{e}} \eps \sg \sqrt{\xx} +  \frac{\lr{2(1 + 2^{-11}) \lr{\ln \frac{8}{\eps}}^{1/\alpha} + 4} \kz \xx}{\min\lrc{2 \az, \logg\lr{\lr{\frac{Kz}{\sg}}^2 \xx} }} + T \\
        & \le (1 + \eps)\sg \sqrt{2\xx} + \frac{7 \lr{\ln \frac{8}{\eps}}^{1/\alpha} \kz \xx}{\min\lrc{2 \az, \logg\lr{\lr{\frac{Kz}{\sg}}^2 \xx} }} + T.
    \end{align*}
    Thus, the first inequality \eqref{res:ben_dim_free} is proven.

    \paragraph*{Bernstein-type bound.} 
    Finally, to prove \eqref{res:ber_dim_free}, we use instead Bernstein-type bounds in \eqref{eq:aux_bounded_dim_free} and \eqref{eq:yet_another_bounde_dim_free}. Then
    \begin{align*}
        \lmax\lr{\Sm_k} & \le \sg \sqrt{2 \xx} + \frac{1}{3} \lr{1 + 2^{-11}} y \xx + \frac{\eps}{e} \sg \sqrt{\xx} + \frac{\eps}{e} \sg \sqrt{\xx} + \frac{3}{4} \kz \xx + T \\
        &\le \sg \sqrt{2 \xx} + \lr{\frac{1 + 2^{-11}}{6 e} + \frac{2}{e}} \eps \sg \sqrt{\xx} +  \lr{\frac{1 + 2^{-11}}{3} \lr{\ln \frac{8}{\eps}}^{1/\alpha} + \frac{3}{4}} \kz \xx + T \\
        & \le (1 + \eps)\sg \sqrt{2\xx} + 2 \lr{\ln \frac{8}{\eps}}^{1/\alpha} \kz \xx + T.
    \end{align*}
\end{proof}

\subsection{Proof of Corollary~\ref{cor:empirical}}
\label{sec:proof_emp_ber}

\begin{lemma}
\label{lemma:bound_on_sigma}
    Under assumptions of Corollary~\ref{cor:empirical},
    for any $\xx > 0$ and $n \ge 2\xx$ it holds, with probability at least $1-3de^{-\xx}$, that
    \[
    \sg \le \left(\hat{\sg} + \norm*{\bar{\X} - \E \X_1}\right) \lr{1 + \frac{2\xx}{3n}}  + \frac{4}{3}K\hat{z}\sqrt{\frac{\xx}{n}}.
    \]
\end{lemma}

\begin{proof}
    The statement is trivial if $\sg \le \hat{\sg}$.
    Now, consider the case $\sg > \hat{\sg}$. 
    In the following, we will construct the bound of the form $\sg \le (\hat{\sg} + \norm*{\bar{\X} - \E \X_1}) (1 + C_1 \frac{\xx}{n}) + C_2\sqrt{\frac{\xx}{n}}$, with $C_1, C_2 > 0$ being some constants. 
        
    To construct this bound, we will use the square-root trick. Let $\Q_i \eqset - (\X_i - \E \X_i)^2$. First, we notice that
    \[
    \hat{\Sg} = \frac{1}{n} \sum \left(\X_i - \E \X_1\right)^2 - (\bar{\X} - \E \X_1)^2 
    = - \frac{1}{n} \sum_i \Q_i - (\bar{\X} - \E \X_1)^2.
    \]
    Thus,
    \[
    \Sg = - \E \Q_1 = \hat{\Sg} + 
    \frac{1}{n} \sum_i \Q_i - \E \Q_1 + (\bar{\X} - \E \X_1)^2.
    \]
    This yields 
    \begin{align}\label{eq:eq_lmin_aux}
       \sg^2 \eqset \lmax(\Sg) &\le \lmax(\hat{\Sg}) + \lmax\left(\frac{1}{n} \sum_i \Q_i - \E \Q_1\right) + \lmax\left((\bar{\X} - \E \X_1)^2\right) \nonumber \\
       &= \hat{\sg}^2 + \lmax\left(\frac{1}{n} \sum_i \Q_i - \E \Q_1\right) + \norm*{\bar{\X} - \E \X_1}^2.
    \end{align}

    Now we have to bound $\lmax\left(\frac{1}{n} \sum_i \Q_i - \E \Q_1\right)$. 
    We will use Theorem~1.4 from \cite{tropp2012user}.
    Its conditions are fulfilled, because all $\Q_i$ are i.i.d., $\Q_i \preccurlyeq 0$, and $\E \Q_i = - \Sg$, thus
    \[
    \lmax\left(\Q_i - \E \Q_1\right) \le \lmax(\Sg) \eqset \sg^2.
    \]
    
    Theorem~1.4 in \citep{tropp2012user} ensures that, with probability at least $1 - d e^{-\xx}$,
    \begin{align}
    \label{eq:aux_tropp_bern}
        \lmax\left(\frac{1}{n} \sum_i \Q_i - \E \Q_1 \right) &\le \sqrt{2 \lmax\left(\E \left(\Q_1 - \E \Q_1 \right)^2 \right) \frac{\xx}{n}} + \frac{2}{3} \sg^2 \frac{\xx}{n} \nonumber \\
        &\le 2 \sg K z \sqrt{\frac{\xx}{n}} + \frac{2}{3} \sg^2 \frac{\xx}{n}.
    \end{align}
    The last inequality follows from Lemma~\ref{lemma:bound_on_4th_moment} that ensures the bound
    \[
    \E (\Q_1 - \E \Q_1)^2 = \E \Q_1^2 - (\E \Q_1)^2 \preccurlyeq \E (\X_1 - \E \X_1)^4 \preccurlyeq \frac{5}{3} (\sg K z)^2 \Ii .
    \]

    Combining \eqref{eq:eq_lmin_aux} and \eqref{eq:aux_tropp_bern}, we get
    \begin{align*}
        \sg^2 &\le \hat{\sg}^2 + \norm*{\bar{\X} - \E \X_1}^2 + 2 \sg K z \sqrt{\frac{\xx}{n}} + \frac{2}{3} \sg^2 \frac{\xx}{n}.
    \end{align*}
    thus
    \[
    \left(1 - \frac{2 \xx}{3 n}\right) \sg^2 
    \le \hat{\sg}^2 + \norm*{\bar{\X} - \E \X_1}^2 + 2 \sg K z \sqrt{\frac{\xx}{n}} .
    \]
    Using a bound on the roots of a square inequality w.r.t.\ $\sigma$, we get
    \begin{align*}
        \sg &\le \sqrt{\frac{1}{1 - \frac{2 \xx}{3 n}} \left(\hat{\sg}^2 + \norm*{\bar{\X} - \E \X_1}^2 \right)} + \frac{2 K z}{1 - \frac{2 \xx}{3 n}} \sqrt{\frac{\xx}{n}} \\
        &\le \frac{\hat{\sg} + \norm*{\bar{\X} - \E \X_1}}{\sqrt{1 - \frac{2 \xx}{3 n}}} + \frac{2 K z}{1 - \frac{2 \xx}{3 n}} \sqrt{\frac{\xx}{n}}.
    \end{align*}

    Lemma's condition ensures $\frac{2\xx}{3n} \le \frac{1}{3}$, thus
    \[
    \sqrt{\frac{1}{1 - \frac{2 \xx}{3 n}}} \le 1 + \frac{2 \xx}{3 n} \quad\text{and}\quad \frac{1}{1 - \frac{2 \xx}{3 n}} \le \frac{3}{2}.
    \]
    Then
    \begin{align*}
        \sg &\le \left(\hat{\sg} + \norm*{\bar{\X} - \E \X_1}\right) \lr{1 + \frac{2 \xx}{3 n}} + \frac{4}{3} K z \sqrt{\frac{\xx}{n}}
    \end{align*}
    Finally, $\hat{\sg} \le \sg$ yields that $z \le \hat{z}$. 
    Thus, we get the result.
\end{proof}

\begin{proof}[Proof of Corollary \ref{cor:empirical}]
    To bound $\norm*{\bar{\X} - \E \X_1}$, we use (in two sides) Corollary~\ref{cor:iid} and Remark~\ref{rem:non-monotone}. 
    This yields that, with probability at least $1 - 2 d e^{-\xx}$,
    \begin{equation}
    \label{eq:bound_on_norm_bar_X}
        \norm*{\bar{\X} - \E \X_1} \le \inf_{\sg' \ge \sg} \lrc{\sg' \sqrt{2 \frac{\xx}{n}} + \frac{3}{4} K z(K, \sg'; \alpha) \frac{\xx}{n}} .
    \end{equation}

    In the case $\sg \le \hat{\sg}$ we immediately obtain
    \[
    \norm*{\bar{\X} - \E \X_1} \le \hat{\sg} \sqrt{2 \frac{\xx}{n}} + \frac{3}{4} K \hat{z} \frac{\xx}{n}.
    \]
    This ensures the result.

    If $\sg \ge \hat{\sg}$, it holds $z \le \hat{z}$. 
    Moreover, we notice that because of Lemma~\ref{lemma:useful_bound}
    \begin{align*}
    \label{eq:bound_on_sg_K}
        \sg^2 &= \lmax(\Sg) = \lmax\lr{\E (\X_1 - \E X_1)^2} \\
        &\le \E \norm{\X_1 - \E \X_1}^2
        \le 2 \lr{\frac{2}{\alpha e}}^{2/\alpha} \norm*{\norm{\X_1 - \E \X_1}}^2_{\pa
        } \le 2 K^2.
    \end{align*}
    This ensures $\hat{\sg} \le \sqrt{2} K \le \sqrt{2} K \hat{z}$.
    
    Now, we use Lemma~\ref{lemma:bound_on_sigma}.
    Keeping in mind Lemma's condition $\frac{\xx}{n} \le \frac{1}{8}$, we get with probability at least $1 - 3 d e^{-\xx}$
    \begin{align*}
        \sg &\le \lr{1 + \frac{2 \xx}{3 n}} \lr{\hat{\sg} + \norm*{\bar{\X} - \E \X_1}} + \frac{4}{3} K \hat{z} \sqrt{\frac{\xx}{n}} \\
        &\le \hat{\sg} + \sqrt{2} K \hat{z} \cdot \frac{2 \xx}{3 n} + \frac{13}{12} \norm*{\bar{\X} - \E \X_1} + \frac{4}{3} K \hat{z} \sqrt{\frac{\xx}{n}} \\
        &\le \hat{\sg} + \frac{13}{12} \norm*{\bar{\X} - \E \X_1} + \left(\frac{4}{3} + \frac{2}{3} \sqrt{2 \frac{\xx}{n}}\right) K \hat{z} \sqrt{\frac{\xx}{n}}\\
        &\le \hat{\sg} + \frac{13}{12} \norm*{\bar{\X} - \E \X_1} + \frac{5}{3} K \hat{z} \sqrt{\frac{\xx}{n}}
    \end{align*}
    
    Now, using the above result and \eqref{res:ber}, we get
    \begin{align*}
        \norm*{\bar{\X} - \E \X_1} & \le \sg \sqrt{2 \frac{\xx}{n}} + \frac{3}{4} K z \frac{\xx}{n} \\
        &\le \hat{\sg} \sqrt{2 \frac{\xx}{n}} + \norm*{\bar{\X} - \E \X_1} \cdot \frac{13}{12} \cdot \sqrt{2 \frac{\xx}{n}} + \frac{5 \sqrt{2}}{3} K \hat{z} \frac{\xx}{n} + \frac{3}{4} K \hat{z} \frac{\xx}{n} \\
        &\le \hat{\sg} \sqrt{2 \frac{\xx}{n}} + \norm*{\bar{\X} - \E \X_1} \cdot \frac{13}{12} \cdot \sqrt{2 \frac{\xx}{n}} + \frac{20 \sqrt{2} + 9}{12} K \hat{z} \frac{\xx}{n}\\
        &\le \hat{\sg} \sqrt{2 \frac{\xx}{n}} + \norm*{\bar{\X} - \E \X_1} \cdot \frac{13}{12} \cdot \sqrt{2 \frac{\xx}{n}} + \frac{7}{2} K \hat{z} \frac{\xx}{n}
    \end{align*}
    Thus, 
    \[
    \lr{1 - \frac{13}{12} \sqrt{2 \frac{\xx}{n}} } \norm*{\bar{\X} - \E \X_1} 
    \le \hat{\sg} \sqrt{2 \frac{\xx}{n}} + \frac{7}{2} K \hat{z} \frac{\xx}{n}.
    \]

    Further, we recall that $n \ge 8 \xx$ and let $t \eqset \sqrt{2 \xx/n} \le \frac{1}{2}$.
    \[
    \lr{1 - \frac{13}{12} t}^{-1} = 1 + \frac{\frac{13}{12} t}{1 - \frac{13}{12} t} 
    \le 1 + \frac{\frac{13}{12}}{1 - \frac{13}{24}} t
    \le 1 + \frac{26}{11} t \le \frac{24}{11}
    \]
    
    Thus, we get
    \begin{align*}
        \norm*{\bar{\X} - \E \X_1} &\le \lr{1 + \frac{26}{11} \sqrt{2 \frac{\xx}{n}}} \hat{\sg} \sqrt{2 \frac{\xx}{n}} + \frac{24}{11} \cdot \frac{7}{2} K \hat{z} \frac{\xx}{n} \\
        &\le \hat{\sg} \sqrt{2 \frac{\xx}{n}} + \frac{52}{11} \cdot \sqrt{2} K \hat{z} \cdot \frac{\xx}{n} + \frac{84}{11} K \hat{z} \frac{\xx}{n} \\
        &\le \hat{\sg} \sqrt{2 \frac{\xx}{n}} + 15 K \hat{z} \frac{\xx}{n}.
    \end{align*}
\end{proof}

\subsection{Proofs of Corollaries~\ref{corollary_maurer1} and~\ref{corollary:maurer2}}
\label{sec:proof_mcdiarmid}

Let $g(x_1, \dots, x_n)$ be a measurable function on $\mathcal{X}^n$, where $\mathcal{X}$ is a measurable space.

Let $I \subset [n]$ and $\bar{I} \eqset [n] \setminus I$.
We denote
\[
g(x_{I}, y_{\bar{I}}) \eqset g(z), 
\quad
z_i \eqset 
\begin{cases}
    x_i &~~\text{if}~i\in I,\\  
    y_i &~~\text{if}~i\in \bar{I}.
\end{cases}
\]

Let $\Pi_n$ be the set of all permutations of $[n]$, $\pi \in \Pi_n \colon [n] \to [n]$. 
We denote
\[
\pi(I) \eqset \{\pi(i): i \in I\}, \quad I \subset [n].
\]
 
\begin{lemma}
\label{lemma:permutations}
    Let $Y = (Y_1, \dots, Y_n)$ be a set of i.i.d.\ random variables on a measurable space $\mathcal{Y}$.
    Let $g_i \colon \mathcal{Y}^n \rightarrow \mathbb{R}_{+}$, $i \in [n]$, be measurable and integrable functions, such that each $g_i(y_1, \dots, y_n)$ does not depend on $y_i$ and
    \[
    \sum_{i \in [n]} g_i(Y) \le M~~\text{a.s.}
    \]
    Let $I \subset [n]$ and define
    \[
    \F_{I} \eqset \sg\lr{\lrc{Y_i: i \in I } }.
    \]
    Let $\Pi_n$ be the set of permutations of $[n]$. 
    Then it holds that
    \[
    \frac{1}{n!} \sum_{\pi \in \Pi_n} \sum^n_{i=1} \E [g_{\pi(i)}(Y) | \F_{\pi([1, i-1])}] 
    \le \frac{n+1}{n} M~~\text{a.s.}
    \]
\end{lemma}

The proof is postponed to the Appendix~\ref{sec:proof_permutations}. 
Now, we are ready to prove the corollaries.

\begin{proof}[Proof of Corollary \ref{corollary_maurer1}]
    We set 
    \[
    X_i \eqset \E [f_i(Y) | \F_i] = \E [f(Y) | \F_i] - \E [f(Y) | \F_{i-1}],
    \]
    so that 
    \[
    \sum^{n}_{i=1} X_i = f(Y) - \E f(Y).
    \]

    Let $\lm > 0$.
    We consider an arbitrary permutation $\pi$ and set
    \[
    V^{\pi}_{i}(\lm) \eqset \ln \E 
    \left[\elr{\lm f(Y) - \E\lrs{ f(Y)| \F_{\pi(i)}}  } | \F_{\pi([1, i-1])} 
    \right].
    \]
    By Proposition~\ref{prop:mart_bound} it holds
    \[
    \E \elr{\lm f(Y) - \sum_i V^{\pi}_{i}(\lm)} \le 1.
    \]
    Thus, Jensen's inequality yields
    \[
    \E\elr{\lm f(Y) - \frac{1}{n!}\sum_{\pi}\sum^n_{i=1} V^{\pi}_i(\lm) } \le \frac{1}{n!}\sum_{\pi} \E \elr{\lm f(y) - \sum^n_{i=1} V^{\pi}_i(\lm)} \le 1.
    \]
    
    Moreover, \eqref{eq:sum_Vi} ensures
    \[
    \sum_i V^{\pi}_i(\lm) \le \frac{\phi(\lm K z)}{(K z)^2} \sum_i \E\lrs{\sg^2_{\pi(i)} \bigg| \F_{\pi(i-1)} } + \frac{2}{K^2}\lr{\phi(\lm Kz) - \frac{\lr{\lm Kz}^2 }{2}} e^{-z^{\alpha}} \sum_i \E \lrs{U^2_{\pi(i)} \bigg| \F_{\pi(i-1)}}.
    \]
    
    Notice that the definitions of $\sg^2_i$ and $U_i$ are equivalent to
    \begin{gather*}
        \sg^2_i = \sg^2_i(Y), \quad \sg^2_i(y) \eqset \E f^2_i(y_1, \dots, y_{i-1}, Y_i, y_{i+1}, \dots, y_n) , \\
        U_i = U_i(Y), 
        \quad
        U_i(y) \eqset \norm*{f_i(y_1, \dots, y_{i-1}, Y_i, y_{i+1}, \dots, y_n)}_{\pa}.
    \end{gather*}
    Now we apply Lemma \ref{lemma:permutations} first setting $g_i = \sg^2_i$, and then setting $g_i = U^2_i$, and get
    \begin{gather*}
        \frac{1}{n!} \sum_\pi \sum_i \E[\sg^2_{\pi(i)} | \F_{\pi(i-1)}] \le \frac{n+1}{n} \sg^2~\text{a.s.},\\
        \frac{1}{n!} \sum_\pi \sum_i \E[U^2_{\pi(i)} | \F_{\pi(i-1)}] \le \frac{n+1}{n} U^2~\text{a.s.}
    \end{gather*}
    Thus,
    \[
    \frac{1}{n!}\sum_{\pi} \sum_{i} V^{\pi}_i(\lm) \le \frac{\phi(\lm K z)}{(K z)^2} \frac{n+1}{n} \sg^2 + \frac{2}{K^2}\lr{\phi(\lm Kz) - \frac{\lr{\lm Kz}^2 }{2}} e^{-z^{\alpha}} \frac{n+1}{n} U^2~\text{a.s.}
    \]
    
    The result follows immediately from Corollary~\ref{cor:1d_case}.
\end{proof}
 
\begin{proof}[Proof of Corollary \ref{corollary:maurer2}]
    We set $f(y) = \norm*{\sum_i y_i}$ and notice, that
    \[
    f_i(Y) = \norm*{\sum_i Y_i} - \E'\norm*{\sum_{j\neq i}Y_j + Y'_i}.
    \]
    Applying the triangle inequality, we get
    \[
    |f_i(Y)| \le \E'\norm*{Y_i - Y'_i}
    \le \norm{Y_i} + \E \norm{Y_i} .
    \]
    Thus,
    \[
    \norm*{f_i(Y) | \F_{-i}}_{\pa} \le \norm*{\norm{Y_i} + \E \norm{Y_i}}_{\pa} \le 
    \norm*{\norm*{Y_i}}_{\pa} + \E \norm{Y_i} \le 2 \norm*{\norm*{Y_i}}_{\pa} \quad \text{a.s.}
    \]
    Similarly, 
    \[
    \E \lrs{f_i^2(Y) | \F_{-i}} \le 4 \E\norm*{Y_i}^2 \quad \text{a.s.}
    \]
    The result follows from Corollary~\ref{corollary_maurer1}.
\end{proof}

\bibliographystyle{unsrtnat}
\bibliography{references}  

\begin{appendices} \section{Tail regimes}

\label{seq:proof_r}

To get~\eqref{eq:subg_r}, we recall that $\alpha z^{\alpha} \ge 4$. 
The result holds due to
\begin{align*}
    \max_k \lmax(\Sm_k) \le \sg \sqrt{2 \xx} + \frac{3}{4} \kz \xx \le \sg \sqrt{2 \xx} \lr{1 + 2 \sqrt{2 \kzss \xx}}
    \le \sg \sqrt{2 \xx} \lr{1 + 2 \sqrt{2 e}} \le 6 \sg \sqrt{2 \xx}.
\end{align*}

Bound \eqref{eq:subpoiss_r} holds due to
\begin{align*}
    \max_k \lmax(\Sm_k) &\le \sg \sqrt{2 \xx} + 4 \kz \frac{\xx}{\ln \lr{\kzss \xx}} \\
    &\le \frac{\sg^2}{\kz}\lr{2 \sqrt{\kzss \xx} + \frac{4 \kzss \xx}{ \ln\lr{\kzss \xx} }} \\
    &\le 8 \frac{\sg^2}{\kz}\frac{\kzss \xx}{\ln\lr{\kzss \xx}} = 8 \frac{\kz \xx}{\ln\lr{\kzss \xx}} .
\end{align*}
The last inequality is due to $\frac{1}{2} \ln x = \ln \sqrt{x} \le \sqrt{x} - 1$.

Bound \eqref{eq:subexp_r} holds due to
\begin{align*}
    \max_k \lmax(\Sm_k) \le \sg \sqrt{2 \xx} + \frac{2 \kz}{\az} \xx 
    \le \frac{\sg^2}{\kz} \lr{\sqrt{2 \kzss \xx} + \frac{1}{\az} 4 \kzss \xx} \le \frac{6 \kz}{\az} \xx.
\end{align*}
The last inequality follows from
\[
\sqrt{2 x} + \frac{4 x}{\az} = \frac{\sqrt{2 x}}{\az} \lr{\az + 2 \sqrt{2 x}} \le \frac{6 x}{\az}.
\]
The last inequality is due to
$\sqrt{x} \ge \elr{\az} \ge e \az$, since $x \ge \elr{2\az}$.

\section{Auxiliary results}\label{sec:auxiliary}

\begin{proof}[Proof of Lemma \ref{lemma:rho}]
    We consider only $x > 0$.
    Recall that $\phi'(t) = e^t - 1 = \phi(t) + t$, thus
    \begin{align*}
        \rho_{\lm,\alpha}'(x) &= \lm \left(\phi'(\lm x) - \lm x\right) \elr{- x^\alpha} - \alpha x^{\alpha - 1} \left(\phi(\lm x) - \frac{(\lm x)^2}{2}\right) \elr{- x^\alpha} \\
        &= \left(\lm \phi(\lm x) - \alpha x^{\alpha - 1} \left(\phi(\lm x) - \frac{(\lm x)^2}{2}\right)\right) \elr{- x^\alpha} \\
        &= \left(\frac{\lm \phi(\lm x)}{\phi(\lm x) - (\lm x)^2 / 2} - \alpha x^{\alpha-1}\right) \rho_{\lm,\alpha}(x).
    \end{align*}
    Since $\phi(\lm x) > \frac{(\lm x)^2}{2}$, $\rho_{\lm,\alpha}(x) > 0$, and hence
    \[
    \sgn \rho_{\lm,\alpha}'(x) = \sgn\left(\frac{\lm \phi(\lm x)}{\phi(\lm x) - (\lm x)^2 / 2} - \alpha x^{\alpha-1}\right) 
    = \sgn\left(\frac{\lm x \phi(\lm x)}{\phi(\lm x) - (\lm x)^2 / 2} - \alpha x^\alpha\right) .
    \]
\end{proof}

\begin{lemma}
\label{lemma:aux_increasing}
    The function $f(t) \eqset \frac{t \phi'(t)}{\phi(t)}$, extended at $0$ by continuity as $f(0) = 2$, is increasing on $\R$
\end{lemma}

\begin{proof}
    To prove the monotonicity of $f$, one has to show that $f$ is continuous at $0$, and $f'(t) > 0$ for all $t \neq 0$.
    
    Note that
    \[
    f(t) = \frac{t (e^t - 1)}{e^t - t - 1} = \frac{t \left(t + o(t)\right) }{\frac{t^2}{2} + o(t^2) } = 2 + o(1), ~~ t\rightarrow 0,
    \]
    hence it can be continuously extended at $0$ as $f(0) = 2$.
    
    The first derivative is
    \[
    f'(t) = \frac{1 + e^{2 t} - e^t (2 + t^2)}{\phi^2(t)} .
    \]
    Consider $u(t) = 1 + e^{2 t} - e^t (2 + t^2)$, its derivative is
    \[
    u'(t) = e^t (2e^t - 2t - 2 - t^2) = 2 e^t \left(\phi(t) - \frac{t^2}{2}\right).
    \]
    By Proposition~\ref{prop:phi_properties}, $\sgn\left(\phi(t) - \frac{t^2}{2}\right) = \sgn(t)$, thus $u$ attains a global minimum at $t = 0$ and $u(t) > u(0) = 0$ for all $t \neq 0$. 
    Therefore, $f'(t) > 0$ for all $t \neq 0$.
\end{proof}

\begin{lemma}\label{lemma:convexity_xphi}
    The function $\upsilon(t)$,
    extended at $0$ by continuity as $\upsilon(0) = 3$, is increasing and strictly convex on $\R$.
    Moreover, for any $t \in \R$
    \[
    \upsilon(t) \le \max\{4, 1.5 t\} .
    \]
\end{lemma}

\begin{proof}
    Note that
    \[
    \upsilon(t) =  \frac{t \phi(t)}{\phi(t) - \frac{t^2}{2}}
    = \frac{t (e^t - 1 - t)}{e^t - 1 - t - \frac{t^2}{2}}.
    \]
    To prove the strict convexity of $\upsilon$, one has to show that $\upsilon$ and $\upsilon'$ are continuous at $0$, and $\upsilon'' (t) > 0$ for all $t \neq 0$.
    Note that 
    \[
    \upsilon(t) = \frac{t \left(\frac{t^2}{2} + o(t^2)\right)}{\frac{t^3}{6} + o(t^3)} = 3 + o(1), \quad t \to 0,
    \]
    hence it can be continuously extended at $0$ with $\upsilon(0) = 3$. 
    The first derivative is 
    \begin{align*}
        \upsilon'(t) &= \frac{\phi(t) + t \phi'(t)}{\phi(t) - \frac{t^2}{2}} - \frac{t \phi(t)^2}{\left(\phi(t) - \frac{t^2}{2}\right)^2} \\
        &= \frac{\left(\frac{t^2}{2} + \frac{t^3}{6} + t (t + \frac{t^2}{2}) + o(t^3)\right) \left(\frac{t^3}{6} + \frac{t^4}{24} + o(t^4)\right) - t \left(\frac{t^2}{2} + \frac{t^3}{6} + o(t^3)\right)^2}{\left(\frac{t^3}{6} + o(t^3)\right)^2} \\
        &= \frac{\left(\frac{3 t^2}{2} + \frac{2 t^3}{3} + o(t^3)\right) \left(\frac{t^3}{6} + \frac{t^4}{24} + o(t^4)\right) - t \left(\frac{t^4}{4} + \frac{t^5}{6} + o(t^5)\right)}{\frac{t^6}{36} + o(t^6)} \\
        &= \frac{\left(\frac{3}{2} \frac{1}{24} + \frac{2}{3} \frac{1}{6} - \frac{1}{6}\right) t^6 + o(t^6)}{\frac{t^6}{36} + o(t^6)} = \frac{t^6 + o(t^6)}{4 t^6 + o(t^6)} = \frac{1}{4} + o(1), \quad t \to 0,
    \end{align*}
    thus it is also continuous at $0$ with $\upsilon'(0) = \frac{1}{4}$.

    The second derivative is 
    \[
    \upsilon''(t) = \frac{t e^t}{4 \left(\phi(t) - \frac{t^2}{2}\right)^3} \underbrace{\left(t^4 + 8 t^2 - 24 + 4 (t^2 + 6) \cosh t - 24 t \sinh t\right)}_{u(t)} .
    \]
    By Proposition~\ref{prop:phi_properties}, $\sgn\left(\phi(t) - \frac{t^2}{2}\right) = \sgn(t)$, thus the first term is positive for any $t \neq 0$.
    
    Now, we show that $u$ is positive as well.
    We explicitly compute 
    \begin{align*}
        u^{(1)}(t) &= 4 t \left(t^2 + 4 - 4 \cosh t + t \sinh t\right),\\
        u^{(2)}(t) &= 4 (3 t^2 + 4 + (t^2 - 4) \cosh t - 2 t \sinh t),\\
        u^{(3)}(t) &= 4 (6 t + (t^2 - 6) \sinh t),\\
        u^{(4)}(t) &= 4 (6 + (t^2 - 6) \cosh t + 2 t \sinh t),\\
        u^{(5)}(t) &= 4 (4 t \cosh t + (t^2 - 4) \sinh t),\\
        u^{(6)}(t) &= 4 t (t \cosh t + 6 \sinh t) = 4 t^2 \cosh t + 24 t \sinh t.
    \end{align*}
    Since $t \sinh t > 0$ and $\cosh t > 1$ for all $t \neq 0$, one immediately obtains that $u^{(6)}(t) > 0$, $t \neq 0$.
    Moreover, $u^{(i)}(0) = 0$ for all $i = 0, \dots, 5$. 
    By Taylor's theorem, this ensures that $u(t) \ge 0$ for all $t$, with the only global minimum $u(0) = 0$. 
    Thus, $\upsilon''(t) > 0$ for all $t \neq 0$, and therefore $\upsilon$ is strictly convex.

    Finally,
    \[
    \lim_{t \to -\infty} \upsilon(t) = \lim_{t \to -\infty} \frac{-t^2 + o(t^2)}{-\frac{t^2}{2} + o(t^2)} = 2,
    \]
    hence $\lim_{t \to -\infty} \upsilon'(t) = 0$ and $f$ is strictly increasing.

    To get the last result, we consider $f(t) = \frac{\phi(t)}{\phi(t) - t^2 / 2}$. 
    Note that $f$ is decreasing by Proposition~\ref{prop:phi_properties} and $\upsilon$ is increasing by Lemma~\ref{lemma:convexity_xphi}. 
    
    If $t_0 = 2.68$,  $\upsilon(t_0) < 1.5$ and $\upsilon(t_0) < 4$. Thus, 
    \[
    \upsilon(t) \le \min\{\upsilon(t_0), t f(t_0)\} < \min\{4, 1.5 t\}.
    \]
\end{proof}

\begin{lemma}
\label{lemma:useful_bound}
    For all $\alpha > 0$, $t \ge 0$, and $p > 0$ it holds
    \[
    t^p \le \left(\frac{p}{\alpha e} \right)^{p/\alpha} e^{t^{\alpha}}.
    \]
    Let $t_1$ be such that $\alpha t^{\alpha}_1 \ge p$. Then for all $t_2 \ge t_1$ 
    \[
    t_2^p e^{-t_2^{\alpha}} \le t_1^p e^{-t_1^{\alpha}}.
    \]
\end{lemma}

\begin{proof}
    Taking the derivative, one gets
    \[
    \left(t^p e^{-t^{\alpha}} \right)' = (p t^{p-1} - \alpha t^{\alpha - 1} t^p) e^{-t^{\alpha}} = (p - \alpha t^\alpha) t^{p-1} e^{-t^{\alpha}} .
    \]
    Thus, the maximum of $t^p e^{-t^{\alpha}}$ is attained at $\alpha t^{\alpha} = p$, where 
    \[
    t^p e^{-t^{\alpha}} = \left( \frac{p}{\alpha}\right)^{p/\alpha}e^{-\frac{p}{\alpha}} = \left(\frac{p}{\alpha e} \right)^{\frac{p}{\alpha}}.
    \]
    The first result follows.
    
    The second result holds because $\left(t^p e^{-t^{\alpha}} \right)' \le 0$ for $\alpha t^{\alpha} \ge p$. 
\end{proof}

\begin{lemma}
\label{lemma:bound_on_e_z}
    Let $\alpha, u, \sg > 0$.
    Then it holds for $z = z(u, \sg; \alpha)$ coming from~\eqref{eq:def_z} that
    \[
    e^{z^{\alpha}} \ge \frac{e^4}{16} \lr{\frac{u z}{\sg}}^2 \ge 3 \lr{\frac{u z}{\sg}}^2 .
    \]
\end{lemma}

\begin{proof}
    Let us define $A \eqset \max\lrc{\frac{u}{\sg},\, 1}$, so that
    \[
    z = \lrs{\frac{4}{\min\{\alpha, 1\}} \ln \frac{e}{\min\{\alpha, 1\}} + 4 \ln A}^{1/\alpha} .
    \]
    
    Note that by Lemma~\ref{lemma:useful_bound},
    \[
    e^{z^{\alpha}} \ge \left(\frac{\alpha e}{4}\right)^{4/\alpha} z^4 .
    \]
    
    First, consider the case $\alpha < 1$. Then
    \[
    e^{z^{\alpha}} = \elr{\frac{4}{\alpha} \ln \frac{e}{\alpha} + 4 \ln A} 
    = \left(\frac{e}{\alpha}\right)^{4 / \alpha} A^4 ,
    \]
    and combining two bounds, we get
    \[
    e^{z^{\alpha}} \ge \left(\frac{\alpha e}{4}\right)^{2/\alpha} z^2 \cdot \left(\frac{e}{\alpha}\right)^{2 / \alpha} A^2
    = \left(\frac{e^2}{4}\right)^{2/\alpha} \left(A z\right)^2 
    \ge \left(\frac{e^2}{4}\right)^{2} \left(A z\right)^2 
    = \frac{e^4}{16} \left(A z\right)^2 .
    \]

    Now, consider the case $\alpha \ge 1$:
    \[
    e^{z^{\alpha}} = \elr{4 + 4 \ln A} = e^4 A^4 ,
    \]
    and thus
    \[
    e^{z^{\alpha}} \ge \left(\frac{\alpha e}{4}\right)^{2/\alpha} z^2 \cdot e^2 A^2
    \ge \left(\frac{e}{4}\right)^{2} e^2 \left(A z\right)^2 
    = \frac{e^4}{16} \left(A z\right)^2 .
    \]
    The claim follows.
\end{proof}
 
\begin{lemma}
\label{lemma:bound_on_4th_moment}
    Fix $\alpha > 0$. 
    Let a random matrix $\X \in \H(d)$ be such that $u \eqset \norm[\big]{\norm{\X}}_{\pa} < \infty$ and $\sg^2 \eqset \lmax(\E \X^2)$. 
    Then for any $z > 0$ such that $\alpha z^\alpha \ge 4$
    \[
    \E \X^4 \preccurlyeq z^2 u^2 \lr{\E \X^2 + 2 z^2 u^2 e^{-z^\alpha} \Ii}.
    \]
    In particular, for $z = z(u, \sg; \alpha)$ defined by~\eqref{eq:def_z},
    \[
    \lmax(\E \X^4) \le \frac{5}{3} (\sg u z)^2.
    \]
\end{lemma}

\begin{proof}
    W.l.o.g.\ we can assume $u = 1$. 
    Notice that
    \begin{align*}
        x^4 &\le x^2 z^2 + x^4 \ind[|x| \ge z] = 
        x^2 z^2 + x^4 e^{-|x|^{\alpha}} e^{|x|^{\alpha}} \ind[|x| \ge z] \\
        &\le x^2 z^2 + z^4 e^{-z^{\alpha}} e^{|x|^{\alpha}} \ind[|x| \ge z],
    \end{align*}
    where the last inequality holds due to Lemma~\ref{lemma:useful_bound} since $\alpha z^{\alpha} \ge 4$.
    Thus, we get
    \begin{equation}
        \E \X^4 \preccurlyeq z^2 \E \X^2 + z^4 e^{-z^{\alpha}} \E e^{\norm{\X}^{\alpha}} \Ii 
        \preccurlyeq z^2 (\E \X^2 + 2 z^2 e^{-z^{\alpha}} \Ii) . 
    \end{equation}
    Finally, for $z = z(u, \sg; \alpha)$, Lemma~\ref{lemma:bound_on_e_z} ensures 
    \[
    \lmax(\E \X^4) \le z^2 \sg^2 + 2 z^4 e^{-z^{\alpha}}
    \le z^2 \sg^2 + \frac{2}{3} z^4 \lr{\frac{\sg}{z}}^2 
    = \frac{5}{3} (\sg z)^2 .
    \]
\end{proof}

\begin{lemma}
\label{lemma:h_bound}
    The inverse function of $h(x) = (x + 1) \ln(x+1) - x$ satisfies
    \begin{equation}
    \label{eq:bound_h_inv}
        h^{-1}(u) \le \sqrt{2 u} + \frac{2u}{\logg 2u}.
    \end{equation}
\end{lemma}

\begin{proof}
    First, consider $0 \le u \le \frac{e^6}{2}$. Then
    \[
    h^{-1}(u) \le \sqrt{2u} + \frac{u}{3} \le \sqrt{2u} + \frac{2u}{\logg 2u} .
    \]
    The first inequality is well-known (see, e.g., Proposition~8 by \citet{sen2018gentle}), and the second trivially follows from the fact that $\logg 2 u < 6$.
    
    Now we consider $u > \frac{e^6}{2}$.
    Since $h(\cdot)$ is increasing, the goal is to check that
    \[
    h\lr{\sqrt{2u} + \frac{2u}{\logg 2u}} - u \geq 0.
    \]
    Notice that
    \[
    \ln \lr{\frac{e^6}{6} + e^3 + 1} \ge 1.
    \]
    Thus, using the definition of $h(\cdot)$ and the inequality $\logg 2u = \ln 2u \ge 6$, we get
    \begin{align*}
        &\lr{\frac{2u}{\ln 2u} + \sqrt{2u} + 1} \ln \lr{ \frac{2u}{\ln 2u} + \sqrt{2u} + 1  } - \frac{2u}{\ln 2u} - \sqrt{2u} - u\\
        &\ge \frac{2u}{\ln 2u} \ln\lr{\frac{2u}{\ln 2u}}  - \frac{2u}{\ln 2u} - u \\
        &= u \lr{ \frac{2}{\ln 2u}  \lr{\ln 2u - \ln\ln 2u}  - \frac{2}{\ln 2u} - 1} = u \lr{1 - \frac{2}{\ln 2u} \lr{ 1 + \ln\ln 2u}} .
    \end{align*}
    Due to the concavity of the logarithm,
    \[
    \ln x \le \ln a + \frac{x - a}{a} \quad \forall x, a > 0,
    \]
    and since $\ln 2 u \ge 6$, we get
    \begin{align*}
        \frac{2}{\ln 2 u} \lr{1 + \ln \ln 2u}
        \le \frac{2}{\ln 2 u} \lr{1 + \ln 6 + \frac{\ln 2 u - 6}{6}}
        = \frac{1}{3} + \frac{2 \ln 6}{\ln 2 u}
        \le \frac{1 + \ln 6}{3} < 1.
    \end{align*}
    The claim follows.
\end{proof}

\begin{lemma}
\label{lemma:bound_on_g_inv}
    Let
    \[
    g_{\lm_0}(t) \eqset \max_{0 \le \lm \le \lm_0}\{\lm t - \phi(\lm) \},
    \]
    and set $x_0 \eqset \lm_0 \phi'(\lm_0) - \phi(\lm_0)$. 
   
    Then, it holds that
    \[
    g^{-1}_{\lm_0}(x) =
    \begin{cases}
        h^{-1}(x), &~\text{if}~ x \le x_0,\\
        t_0 + \frac{x - x_0}{\lm_0} \le \frac{2}{\lm_0} x, &~\text{if}~ x > x_0.
    \end{cases}
    \]
\end{lemma}

\begin{proof}
    Notice that $\phi(\cdot)$ is strictly convex. 
    Thus $u(\lm) \eqset \lm t - \phi(\lm)$ is strictly concave and its $\max$ is unique. 
    Consider
    \[
    u'(\lm) = t - \phi'(\lm) = t + 1 - e^{\lm} = 0.
    \]
    Thus, the global $\max$ is attained at $\ln(t + 1)$. 
    
    Taking into account the condition $0 \le \lm \le \lm_0$, we get 
    \[
    g(t) = \lm^* t - \phi(\lm^*), 
    \quad
    \lm^* = \min\{\lm_0, \ln(t + 1) \}.
    \]
    The critical point is $t_0 = e^{\lm_0} - 1 = \phi'(\lm_0)$. 
    
    If $t \le t_0$,
    \[
    g(t) = t\ln(t+1) - \phi(\ln(t+1)) = h(t).
    \]
    Thus, for all $x \le x_0$, $x_0 \eqset h(t_0)$,
    \[
    g^{-1}(x) = h^{-1}(x).
    \]

    We also notice that substituting $t_0 = e^{\lm_0} - 1 = \phi'(\lm_0)$ to $h(t_0)$, one gets
    \[
    x_0 \eqset h(t_0) = \lm_0 \phi'(\lm_0) - \phi(\lm_0).
    \]
    
    Now consider $t > t_0$,
    \[
    g(t) = \lm_0 t - \phi(\lm_0).
    \]
    This yields
    \[
    g^{-1}_{\lm_0}(x) = \frac{\phi(\lm_0) + x}{\lm_0} 
    \le \frac{\phi(\lm_0) x / x_0 + x}{\lm_0} = \frac{\phi(\lm_0) + x_0}{x_0} \frac{x}{\lm_0}.
    \]

    Finally, notice that 
    \[
    \frac{x_0}{\phi(\lm_0) + x_0} = \frac{\lm_0 \phi'(\lm_0) - \phi(\lm_0)}{\lm_0 \phi'(\lm_0)} = 1 - \frac{\phi(\lm_0)}{\lm_0 \phi'(\lm_0)} \ge \frac{1}{2},
    \]
    the last inequality follows from the bound $\frac{\lm_0 \phi'(\lm_0)}{\phi(\lm_0)} \ge 2$ due to Lemma \ref{lemma:aux_increasing}.
\end{proof}

\section{Proof of Lemma~\ref{lemma:permutations}}
\label{sec:proof_permutations}

\begin{proof}[Proof of Lemma~\ref{lemma:permutations}]
    Let $Y' = (Y'_1, \dots, Y'_n)$ be an independent copy of $Y$ (all $Y'_i$ are i.i.d.). 
    First, we notice that
    \[
    \E_{Y} \left[g_{\pi(i)}(Y) | \F_{\pi([1, i-1])}\right] 
    = \E_{Y'} g_{\pi(i)}\left(Y_{\pi([1, i-1])}, Y'_{\pi([i, n])} \right).
    \]
    This yields
    \[
    \frac{1}{n!} \sum_{\pi\in\Pi_n} \sum^n_{i=1} \E_{Y}[g_{\pi(i)}(Y) | \F_{\pi([1, i-1])}] 
    = \E_{Y'} \frac{1}{n!} \sum^n_{i=1} \sum_{\pi\in\Pi_n} g_{\pi(i)}(Y_{\pi([1, i-1])}, Y'_{\pi([i, n])}).
    \]

    Second, we notice that for $j \in I \subset [n]$
    \[
    g_j(Y_{I}, Y'_{\bar{I}}) = g_{j}(Y_{I \setminus \{j\}}, Y'_{\bar{I}\bigcup \{j\} }),
    \]
    since $g_{j}(x_1, \dots, x_n)$ does not depend on $x_{j}$.
    Thus, we get for a fixed $i \in [n]$
    \begin{align*}
    \label{eq:just_a_bound}
         \sum_{\pi \in \Pi_n} g_{\pi(i)}\left(Y_{\pi([1, i])}, Y'_{\pi([i+1, n])} \right)
         &= (i-1)! (n-i)! \sum_{I \subset [n]:~|I|=i} \sum_{j \in I} g_j(Y_{I}, Y'_{\bar{I}}) \\
         & = (i-1)! (n-i)! \sum_{J \subset [n]:~|J| = i-1} \sum_{j \notin J} g_j(Y_{J}, Y'_{\bar{J}}).
    \end{align*}
    
    Now let $\alpha_i \eqset \frac{i}{n}$ for $i\in [n]$. 
    Combining the above results, we get
    \begin{align*}
        \frac{1}{n!} & \sum_{\pi \in \Pi_n} \sum^n_{i=1} \E \lrs{g_{\pi(i)}(Y) \middle| \F_{\pi([1, i-1])}} \\
        &= \frac{1}{n!} \sum^n_{i=1} (i-1)!(n-i)! \lr{\alpha_i \sum_{|I| = i} \sum_{j\in I} g_j(Y_{I}, Y'_{\bar{I}}) + (1-\alpha_i) \sum_{ |I| = i-1 } \sum_{j \neq I} g_j(Y_{I}, Y'_{\bar{I}})} \\
        & = \frac{1}{n!} \sum^n_{i=1} \alpha_i (i-1)!(n-i)! \sum_{|I|=i} \sum_{j\in I}g_j(Y_{I}, Y'_{\bar{I}}) + \frac{1}{n!} \sum^{n-1}_{i=0}(1-\alpha_{i+1}) i! (n-i-1)! \sum_{|I|=i}\sum_{j\notin I}g_j(Y_{I}, Y'_{\bar{I}}).
    \end{align*}
    
    Now we notice that
    \[
    \alpha_i (i-1)! (n-i)! = \frac{i! (n-i)!}{n}, 
    \quad
    (1 - \alpha_{i+1}) i! (n-i-1)! \le \frac{i! (n-i)!}{n} .
    \]
    Thus,
    \begin{align*}
        \frac{1}{n!} &\sum^n_{i=1} \alpha_i (i-1)!(n-i)! \sum_{|I|=i} \sum_{j \in I}g_j(Y_{I}, Y'_{\bar{I}}) + \frac{1}{n!} \sum^{n-1}_{i=0} (1-\alpha_{i+1}) i! (n-i-1)! \sum_{|I|=i} \sum_{j \notin I}g_j(Y_{I}, Y'_{\bar{I}}) \\
        &\le \sum^n_{i=1} \frac{i! (n-i)!}{n \cdot n!} \lr{\sum_{|I|=i} \sum_{j \in I} g_j(Y_{I}, Y'_{\bar{I}}) + \sum_{|I|=i} \sum_{j \notin I} g_j(Y_{I}, Y'_{\bar{I}})} \\
        &= \sum_{i=0}^n \frac{i! (n-i)!}{n \cdot n!} \sum_{|I|=i} \sum_{j=1}^n g_j(Y_{I}, Y'_{\bar{I}}).
    \end{align*}
    Recall that by the Lemma's condition it holds that $\sum_{j=1}^n g_j(Y_{I}, Y'_{\bar{I}}) \le M$ a.s. 
    Further, the number of subsets of cardinality $i$ is
    $|\{I \subset [n] : |I| = i\}| = \binom{n}{i} = \frac{n!}{i! (n-i)!}$. 
    Thus, we get
    \[
    \frac{1}{n!} \sum_{\pi \in \Pi} \sum^n_{i=1} \E \lrs{g_{\pi(i)}(Y) \middle| \F_{\pi([1, i-1])}}
    \le \sum_{i=0}^n \frac{i! (n-i)!}{n \cdot n!} \binom{n}{i} M = \frac{n+1}{n} M .
    \]
\end{proof}

\end{appendices}

\end{document}

%% file: mydef.tex
\newcommand{\eqset}{\coloneqq}

\renewcommand{\Gamma}{\varGamma}
\renewcommand{\Pi}{\varPi}
\renewcommand{\Delta}{\varDelta}
\renewcommand{\Psi}{\varPsi}
\renewcommand{\Phi}{\varPhi}
\renewcommand{\Theta}{\varTheta}
\renewcommand{\Omega}{\varOmega}
\renewcommand{\Xi}{\varXi}
\renewcommand{\Upsilon}{\varUpsilon}

\let\norm\relax
\DeclarePairedDelimiter{\norm}{\lVert}{\rVert}
\let\abs\relax
\DeclarePairedDelimiter{\abs}{\lvert}{\rvert}

\newcommand*{\R}{\mathbb{R}}

\renewcommand*{\H}{\mathbb{H}}

\DeclareMathOperator{\diag}{diag}
\DeclareMathOperator{\tr}{tr}

\renewcommand*{\P}{\mathbb{P}}
\DeclareMathOperator{\E}{\mathbb{E}}

\newcommand*{\F}{\mathscr{F}}

\DeclareMathOperator{\ind}{\mathbb{I}}
\DeclareMathOperator{\sgn}{sign}

\newcommand*{\eps}{\varepsilon}
\newcommand*{\lmax}{\lambda_{\max}}

\newcommand*{\Q}{\mathbf{Q}}

\renewcommand{\bar}[1]{\overline{#1}}
\renewcommand{\tilde}[1]{\widetilde{#1}}

\newcommand{\pa}{\psi_{\alpha}}
\newcommand{\lm}{\lambda}
\newcommand{\sg}{\sigma}
\newcommand{\lr}[1]{\left(#1 \right)}
\newcommand{\lrc}[1]{\left\{#1 \right\}}
\newcommand{\lrs}[1]{\left[#1 \right]}
\newcommand{\elr}[1]{\exp\left\{#1 \right\}}

\newcommand{\x}{\mathrm{x}}

\newcommand{\az}{\alpha z^{\alpha}}
\newcommand{\azh}{2 \alpha z^{\alpha}}

\newcommand{\kz}{Kz}

\newcommand{\kzss}{\left( \frac{Kz}{\sigma}\right)^2}

\newcommand{\xx}{\mathrm{x}}
\newcommand{\logg}{\hspace{0.2em}\underline{\log}\hspace{0.2em}}

\newcommand{\X}{\mathbf{X}}
\newcommand{\A}{\mathbf{A}}
\newcommand{\Y}{\mathbf{Y}}
\newcommand{\Ii}{\mathbf{I}}
\newcommand{\Z}{\mathbf{Z}}
\newcommand{\Hh}{\mathbf{H}}
\newcommand{\V}{\mathbf{V}}
\newcommand{\Sg}{\boldsymbol{\Sigma}}

\newcommand{\Sm}{\mathbf{S}}